\documentclass[10pt]{amsart}
\usepackage{amsmath, amsthm, amscd, amssymb}
\usepackage{pinlabel}
\usepackage{hyperref}
\usepackage{subfig}

\newtheorem{theorem}{Theorem}[section]
\newtheorem{lemma}[theorem]{Lemma}
\newtheorem{proposition}[theorem]{Proposition}

\newtheorem*{theorem*}{Theorem}
\newtheorem*{proposition*}{Proposition}
\newtheorem*{conjecture*}{Conjecture}
\newtheorem*{disk_orbifolds_thm}{Theorem~\ref{thm:disk_orbifolds}}
\newtheorem*{genus_orbifolds_thm}{Theorem~\ref{thm:genus_orbifolds}}

\theoremstyle{definition}

\theoremstyle{remark}
\newtheorem{remark}[theorem]{Remark}
\newtheorem{example}[theorem]{Example}

\def\Z{\mathbb Z}
\def\R{\mathbb R}

\def\H{\mathbb H}
\def\N{\mathbb N}

\def\SS{\Sigma}

\def\scl{\ensuremath{\textnormal{scl}}}

\def\PSL{\textnormal{PSL}}

\def\hs1t{\ensuremath{\textnormal{Homeo}^+(S^1)^\sim}}

\makeindex

\begin{document}

\title{Stable immersions in orbifolds}

\author{Alden Walker}
\address{Department of Mathematics \\ University of Chicago \\
Chicago, IL  60637}
\email{akwalker@math.uchicago.edu}

\begin{abstract}
We prove that in any hyperbolic orbifold with one boundary component, 
the product of any hyperbolic fundamental group element with a sufficiently 
large multiple of the boundary is represented by a geodesic loop 
that virtually bounds an immersed surface.  In the case that 
the orbifold is a disk, there are some conditions.  
Our results generalize work of Calegari-Louwsma and resolve a 
conjecture of Calegari.
\end{abstract}

\maketitle


\section{Introduction}

It is an interesting and important problem to understand 
which curves on a surface bound an immersed subsurface.  
This paper addresses a question in this area primarily motivated by 
stable commutator length ($\scl$) and quasimorphisms, and in this 
introduction, we provide some background.  However, the 
main theorems are concerned only with immersions, so the 
reader can safely skim the portions of this introduction concerned 
with $\scl$  
and retain a logically complete (though less colorful!) picture. 
For a more thorough $\scl$ background, especially as it relates 
to quasimorphisms and immersions, see~\cite{C_faces} and~\cite{CL}.

\subsection{Orbifolds}

Recall that an \emph{orbifold} is a 
space locally modeled on Euclidean space modulo finite groups of 
isometries.  See \cite{Thurston} for background.
In this paper, we will be concerned only with good orbifolds 
with a hyperbolic structure.  
By a \emph{hyperbolic orbifold} $\Sigma$, we mean an orientable orbifold
which arises as the quotient of hyperbolic space $\H^2$ by a finitely generated 
discrete subgroup $\Gamma \subseteq \PSL(2,\R)$ 
such that $\Gamma$ acts properly on $\H^2$ and $\Gamma \backslash \H^2$ is 
finite-volume.  Thus, $\H^2$ is the universal 
orbifold cover of $\Sigma$, and $\Gamma$ is identified with $\pi_1(\Sigma)$.
We will use this notation throughout the paper.

Geometrically, a hyperbolic orbifold is a hyperbolic 
surface with finitely many cone points and cusps.  We will 
be interested in how the hyperbolic structure can inform 
topological properties of $\Sigma$, so it is useful to 
also have a topological picture.  Topologically, a
hyperbolic orbifold is an orientable surface
with finitely many points with a nontrivial structure (isotropy) group 
(which is always a finite cyclic group), and finitely many points removed.
The \emph{underlying space} of an orbifold $\Sigma$ 
is the topological space $\Gamma\backslash \H^2$ with the orbifold 
structure forgotten.  A \emph{disk orbifold} is a hyperbolic orbifold 
whose underlying space is a disk; that is, $\Sigma$ is topologically a 
sphere with cone points and a single removed point.  There are various 
notations for orbifolds, which we will mostly avoid; however, 
following \cite{CL}, we will refer to disk orbifolds 
with two cone points of orders $p$ and $q$ (and one cusp) as $(p,q,\infty)$ 
orbifolds.  We clarify that this notation does not refer to a triangle 
group, which contains an orientation-reversing reflection.
As an example, if we set $\Gamma = \PSL(2,\Z)$, we get 
the modular orbifold, which is a $(2,3,\infty)$ orbifold.

While a hyperbolic orbifold technically has cusps instead of boundary, 
it still has natural boundary elements of the fundamental group, as follows.
A small loop around a cusp gives a conjugacy class
in the fundamental group $\pi_1(\Sigma)$.  In the identification $\Gamma = \pi_1(\Sigma)$, 
this class is identified with the (parabolic) stabilizers of the preimages
of the cusp in $\H^2$.  Abusing notation, we will use $\partial \Sigma$ 
to mean either the union of the small loops around the cusps of $\Sigma$ 
or the union of the associated conjugacy classes in $\Gamma$, 
and we will refer to the loops as \emph{boundary components}, or boundary loops, of $\SS$.
As noted above, topologically, a hyperbolic orbifold is homotopy 
equivalent to an orbifold in which the cusps have been replaced 
with honest boundaries, motivating this nomenclature.

\begin{remark}
Just as in \cite{CL}, the results and proofs in this paper apply equally 
well to hyperbolic orbifolds with geodesic boundaries instead of cusps, 
in which case the universal orbifold cover is not the entire 
hyperbolic plane.  For simplicity, however, we will always 
use the definition of hyperbolic orbifold above.
\end{remark}

\subsection{Immersions}

Let $S$ be a smooth surface, possibly with boundary and 
removed points, and let $\SS$ be a hyperbolic orbifold.  
If $f:S \to \SS$ is a continuous map, then it lifts 
to a map $\widetilde{f}:\widetilde{S} \to \widetilde{\SS} = \H^2$ between 
universal covers.  We say that $f$ is an \emph{immersion}
if $\widetilde{f}$ is. Note that this is equivalent to 
saying that $f$ is an immersion away from the preimages of 
the cone points of $\SS$, and at the preimages of a cone 
point with angle $2\pi/n$, $f$ 
has branch points of order exactly $n$.
We will only be interested in orientation-preserving 
immersions, although the techniques in this paper apply to orientation-reversing 
ones as well.

If $\SS$ is a hyperbolic orbifold with 
fundamental group $\Gamma$, 
and $g \in \Gamma$ is a hyperbolic element, then $g$ is 
represented by a unique geodesic $\gamma \in \SS$.
We say that $\gamma$ (or $g$) \emph{bounds an immersed surface}
if there is an oriented surface $S$ and an orientation-preserving 
immersion $f:S \to \SS$ such that $f(\partial S) = \gamma$ 
(as oriented $1$-manifolds).  We say that 
$\gamma$ (or $g$) \emph{virtually bounds an immersed surface} 
if there is an oriented surface $S$ and an 
orientation-preserving immersion $f : S \to \SS$ such that 
$f|_{\partial S}$ is a covering map $\partial S \to \gamma$.
There are examples of curves on surfaces which do not bound an 
immersed surface but do virtually bound an immersed surface.
We emphasize that a group element $g \in \Gamma$ 
only (virtually) bounds an immersed surface if the surface 
boundary maps to the \emph{geodesic} representative of $g$.

\subsection{\texorpdfstring{$\scl$}{scl} and stability}
One can make the analogous definition of virtually 
bounding an immersed surface for any homologically trivial 
$1$-chain $C \in B_1(\Gamma)$, and in 
\cite{C_faces}, Calegari shows the following stability 
theorem, simplified slightly here.

\begin{theorem*}[\cite{C_faces}, Theorem~C]
Let $\SS$ be a compact, connected, orientable surface with 
boundary and $C \in B_1^H(\Gamma)$, where $\Gamma = \pi_1(\SS)$.  
Then for all sufficiently large $n$, 
the chain $n\partial \SS + C$ virtually bounds an 
immersed surface.
\end{theorem*}

This theorem applies to orbifolds via lifting.
Here $B_1^H(\Gamma) = B_1(\Gamma) / \langle g^n = ng, hgh^{-1}=g\rangle$ 
is the space of \emph{homogenized} $1$-chains, which is more natural 
from the perspective of virtual immersions and $\scl$.
Stable commutator length is a norm on the vector 
space $B^H_1(\Gamma)$ (see~\cite{C_scl} for background), 
and in \cite{C_scallop}, Calegari shows 
that the $\scl$ norm ball is polyhedral, in that its restriction 
to any finite-dimensional subspace is a rational finite-sided polyhedron.
In \cite{C_faces}, Calegari shows that 
there is a distinguished codimension-one so-called \emph{geometric} face of the $\scl$ 
norm ball associated to the realization of the abstract group 
$\Gamma$ as the fundamental group of the orbifold $\SS$.  
This geometric face is dual to 
the rotation quasimorphism on $\Gamma$ induced by the 
circle action at infinity coming from the identification 
of $\Gamma = \pi_1(\SS) \subseteq \PSL(2,\R)$.
The $1$-chains projectively contained in the 
geometric face are exactly those which 
virtually bound immersed surfaces, so 
Theorem~C is the main technical result showing that 
the geometric face is codimension-one.

So \cite{C_faces} provides a fundamental connection between 
(virtual) immersions, $\scl$, and rotation quasimorphisms.  
Computer experiments led to the following conjecture:
\begin{conjecture*}[\cite{C_faces}, Conjecture~3.16]
Let $F= \langle a,b \rangle$ be a free group of rank $2$.  
Let $w \in [F,F]$ be any homologically trivial word.  
Then for sufficiently large $n$, $w[a,b]^n$ virtually 
bounds an immersed surface in the realization of $F$ as 
the fundamental group of the hyperbolic once-punctured 
torus with boundary $[a,b]$. 
\end{conjecture*}
Note that \cite{C_faces}, Conjecture~3.16 and 
\cite{C_faces}, Theorem~C involve two similar, but 
definitely distinct, notions of stability; in Theorem~C, 
we are taking a formal sum with a multiple of the boundary, and 
in \cite{C_faces}, Conjecture~3.16, we are multiplying by it.

In~\cite{CL}, Calegari and Louwsma prove the 
analog of \cite{C_faces}, Conjecture~3.16 for $(2,p,\infty)$ 
orbifolds:
\begin{theorem*}[\cite{CL}, Theorem~3.1]
Let $\SS$ be a $(2,p,\infty)$ orbifold with 
boundary loop $b$, and let $w \in \pi_1(\SS)$ be any hyperbolic 
element.  Then for all sufficiently large $n$, 
$wb^n$ virtually bounds an immersed surface in $\SS$.
\end{theorem*}

There is a potential ambiguity here, in that the boundary loop $b$ is not an element but a 
conjugacy class.  But note that if the theorem holds for some representative 
of this conjugacy class, it holds for all of them, because the effect of changing 
representatives is essentially to change the word $w$.  This is also true of the next 
theorem.

\subsection{Results}

In this paper, we generalize \cite{CL}, Theorem~3.1 
(see Remark~\ref{rem:generalization_of_CL}, which addresses 
the issue of two vs three orbifold points) with the 
following theorem.

\begin{disk_orbifolds_thm}
Let $\SS$ be a hyperbolic orbifold whose underlying topological space is a disk and 
which has at least three orbifold points, of orders $\{o_j\}_{j=0}^{J-1}$. 
Let $w \in \Gamma = \pi_1(\SS)$ be any hyperbolic element, and let $b \in \Gamma$ be the 
boundary loop of $\SS$.  Then there exists $N \in \N$ so that 
for all $n\ge 0$, the loop $wb^{N+ng}$ virtually 
bounds an immersed surface, where $g = \gcd(o_0-1, \ldots, o_{J-1}-1)$.
\end{disk_orbifolds_thm}

We also resolve \cite{C_faces}, Conjecture~3.16, even in the 
presence of orbifold points

\begin{genus_orbifolds_thm}
Let $\SS$ be a hyperbolic orbifold with one boundary component and 
with genus at least $1$.  Let $\Gamma = \pi_1(\SS)$ with $b = \partial \SS \in \Gamma$,
and let $w \in \Gamma$ be hyperbolic so that some power of $w$ is homologically trivial.  
Then there exists $N\in \N$ so that 
for all $n\ge 0$, the loop representing $wb^{N+n}$ virtually bounds 
an immersed surface.
\end{genus_orbifolds_thm}

In the course of proving these theorems, we also give 
a useful combinatorial certificate (Proposition~\ref{prop:immersion})  
that a surface map into 
an orbifold $\SS$ is homotopic to an immersion with geodesic boundary, and hence 
a certificate that a collection of 
words in $\Gamma$ bounds an immersed surface.

\subsection{Outline}

In Section~\ref{sec:orbifolds}, we review cyclic orders and 
realizations of a group as the fundamental group of a hyperbolic 
orbifold.  In Section~\ref{sec:immersion},
we give a combinatorial parameterization of surface maps into orbifolds.
In Section~\ref{sec:stability}, we prove our main theorems.

\subsection{Acknowledgments}
We wish to thank Danny Calegari, Joel Louwsma, Neil Hoffman, 
and the anonymous referee.  Alden Walker 
was supported by NSF grant DMS~1203888.

\section{Hyperbolic orbifolds as realizations}
\label{sec:orbifolds}

\subsection{Cyclic orders}
Informally, a cyclic order on a set $S$ is an arrangement of $S$ 
around a circle, and there are several equivalent ways of formalizing 
this.  We define a \emph{cyclic order} on a set $S$ to be a function 
$O: S\times S \times S \to \{-1, 0, 1\}$ which says whether 
a triple of elements is positively or negatively ordered 
(or $0$, if not all elements in the triple are distinct).
A cyclic order must satisfy a compatibility condition on all $4$-tuples 
of elements; namely if $O(x,y,z)=O(w,x,z)=1$, then 
$O(x,y,w)=O(y,z,w)=1$, the idea being that if we know $(x,y,z)$ and 
$(w,x,y)$ are positively ordered, then we can conclude that 
the four elements are arranged in the order $[x,y,z,w]$, and 
the cyclic order $O$ must respect this.  As a shorthand for 
the function $O$, we will write cyclic orders in square brackets, 
as above, recording the (ordered) arrangement of the elements 
around a circle.  A cyclic order given in square brackets is invariant 
under cyclic permutations of the list, and the function 
$O(x,y,z)$ can be computed by rotating the list so that $x$ 
is first; the value is then $1$ if $y$ comes before $z$ and $-1$ 
otherwise.
 
If $T \subseteq S$ and $O_T$ and $O_S$ are cyclic orders on $T$ and $S$, 
respectively, then we say that $O_T$ and $O_S$ are \emph{compatible} 
if $O_S|_T = O_T$.  If $T$ and $S$ are finite and the orders 
are written as cyclic lists with square brackets, 
then $O_T$ and $O_S$ are compatible if $O_T$ is obtained 
from $O_S$ by simply removing the elements of $S \setminus T$.
For example, the orders $[a,c,b]$ and $[a,c,d,b]$ are compatible.
See~\cite{C_foliations}, Chapter~2.

\subsection{Realizations}

Let $\SS$ be a hyperbolic orbifold with one cusp, and let $\Gamma = \pi_1(\SS)$ 
be its fundamental group, so $\Gamma \subseteq \PSL(2,\R)$ is a group 
of isometries. 
We now find a nice generating set 
for $\Gamma$.  Let $R$ be a fundamental domain for the action, 
which is a polygon in $\H^2$ with some 
ideal vertices.  The group $\Gamma$ is a free product of 
cyclic groups, which we write 
$\Gamma = \left(*_{i=0}^{I-1} Z_i\right) * \left(*_{j=0}^{J-1} C_j\right)$, 
where each $Z_i$ is infinite cyclic (i.e. a copy of $\Z$) and 
is generated by $z_i$, and each $C_j$ is finite cyclic and 
is generated by $c_j$ with order $o_j$.  The $z_i$ are hyperbolic, and the $c_j$ 
are elliptic.  After conjugation, we may assume that 
the axes of the $z_i$ all pass through $R$, and the 
fixed points of the $c_j$ are all vertices of $R$.  
It is possible that there is more than one choice 
for the $c_j$, since a single orbifold point 
may appear multiple times as a vertex of $R$.  Any of the 
options will work.  We may also assume that 
all the $c_j$ rotate counterclockwise.  
We will always write words in $\Gamma$ using positive 
powers of the $c_j$.  For a given word $w$ in the given generators of $\Gamma$, 
we will call a specific generator at a specific location 
in $w$ a \emph{letter}, and we'll denote the letter in $w$ 
at position $k$ by $w_k$, with indices starting at $0$.

We call the orbifold $\SS$ together with the generating set 
$\{c_0,\ldots, c_{J-1},z_0,\ldots, z_{I-1}\}$ chosen as above a 
\emph{realization} of the abstract group 
$\left(*_{i=0}^{I-1} Z_i\right) * \left(*_{j=0}^{J-1} C_j\right)$, 
and we will always assume that our hyperbolic orbifolds 
come with such a generating set.

For each $z_i$, mark the intersections of the hyperbolic axis of 
$z_i$ with the boundary of $R$: the initial intersection with $z_i^{-1}$ 
and the terminal intersection with $z_i$.  Also mark the elliptic 
fixed point of $c_j$ by $c_j$.  Reading the boundary 
of $R$ counterclockwise, this induces a cyclic order 
on the set of generators
$S = \{z_0, z_0^{-1}, \ldots, z_{I-1}, z_{I-1}^{-1}, c_0, \ldots, c_{J-1}\}$.
Note that the set $S$ contains each $z_i$ and its inverse, but 
only the positive power of $c_j$.  We'll denote the cyclic 
order on $S$ by $O_\SS$.  

It may seem as though there is potential ambiguity in the cyclic order $O_\SS$, 
because some 
elliptic $c_j$ may have been associated with multiple vertices of $R$, 
and we chose the vertex to be labeled arbitrarily.  However, 
note that if we were to choose a different vertex, 
that would correspond to choosing a different (conjugate) 
generator in $\Gamma$, so while we would get a different cyclic order, 
it's also a genuinely different identification with 
$\left(*_{i=0}^{I-1} Z_i\right) * \left(*_{j=0}^{J-1} C_j\right)$.

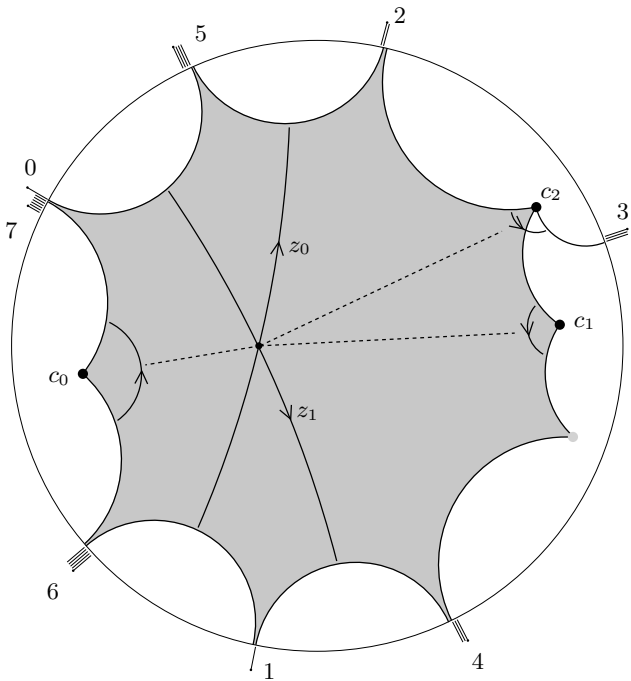
\begin{figure}[[ht]
\begin{center}
\labellist
\small\hair 2pt
 \pinlabel {$z_0$} at 181 266
 \pinlabel {$z_1$} at 186 162
 \pinlabel {$c_0$} at 30 184
 \pinlabel {$c_1$} at 360 219
 \pinlabel {$c_2$} at 339 300
 \pinlabel {$0$} at 12 317
 \pinlabel {$1$} at 162 0
 \pinlabel {$2$} at 244 412
 \pinlabel {$3$} at 384 289
 \pinlabel {$4$} at 293 6
 \pinlabel {$5$} at 119 401
 \pinlabel {$6$} at 26 50
 \pinlabel {$7$} at 0 277
\endlabellist
\includegraphics[scale=0.6]{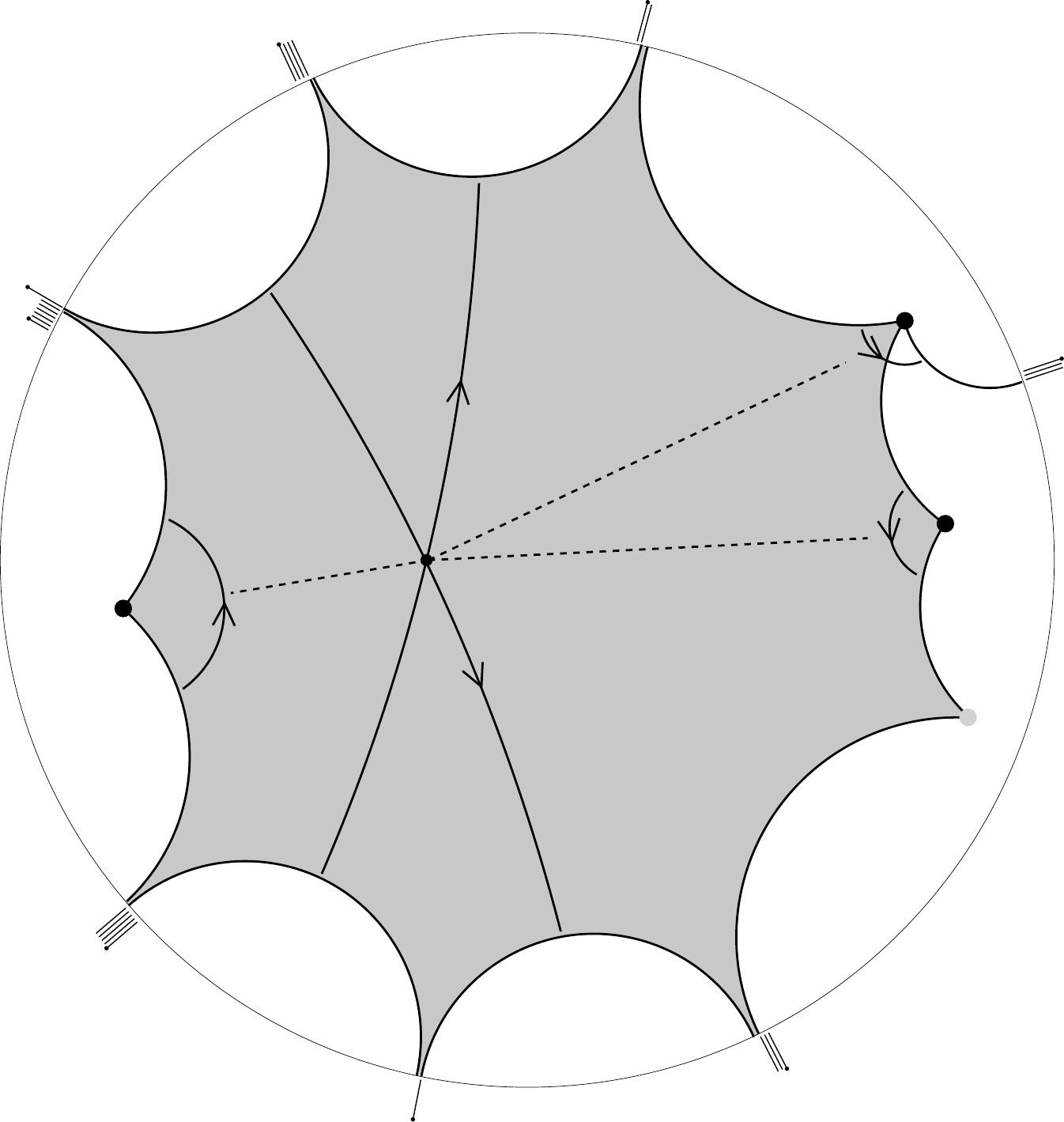}
\caption[A fundamental domain for an orbifold]{A fundamental domain 
for an orbifold, as 
described in Example~\ref{ex:fundamental_domain}.}\label{fig:fundamental domain}
\end{center}
\end{figure}

\begin{figure}[[ht]
\begin{center}
\labellist                                                                                                                                  
\small\hair 2pt                                                                                                                             
 \pinlabel {$z_0$} at 179 291                                                                                                             
 \pinlabel {$z_0^{-1}$} at 159 257                                                      
 \pinlabel {$z_1^{-1}$} at 190 2
 \pinlabel {$z_1$} at 149 34                                                                                                              
 \pinlabel {$c_0$} at 115 155                                                                                                             
 \pinlabel {$c_1$} at 328 99                                                                                                              
 \pinlabel {$c_2$} at 331 215                                                                                                             
\endlabellist  
\includegraphics[scale=0.6]{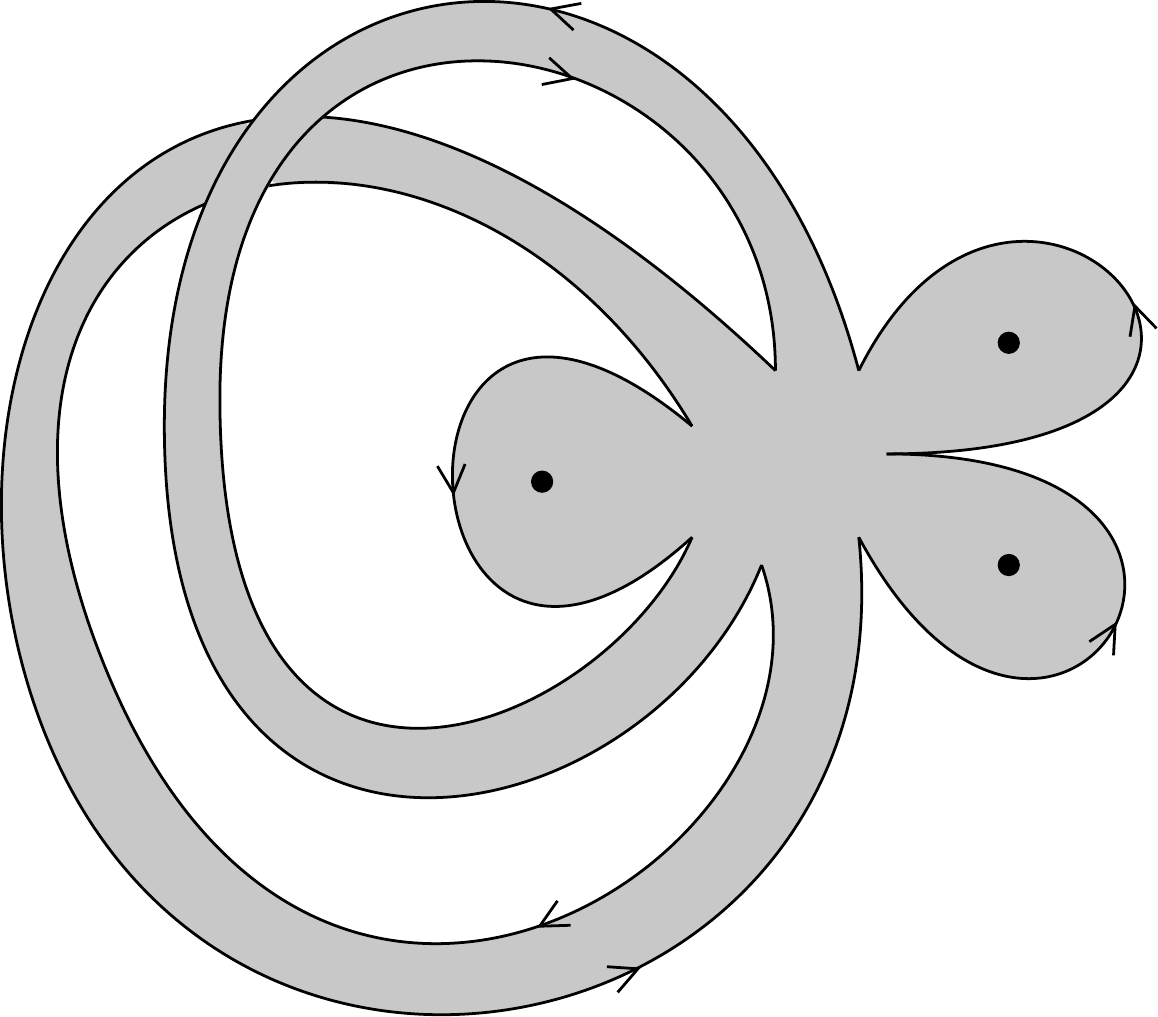}
\caption[A combinatorial version of the orbifold]{The 
orbifold from Figure~\ref{fig:fundamental domain}, topologically, with the 
loops representing the generators of $\Gamma$.  Note the boundary 
word is easy to read off.  The cyclic order $O_\SS$ on generators 
is the cyclic order around the central part of the surface. See Example~\ref{ex:fundamental_domain}.}\label{fig:topological_type}
\end{center}
\end{figure}
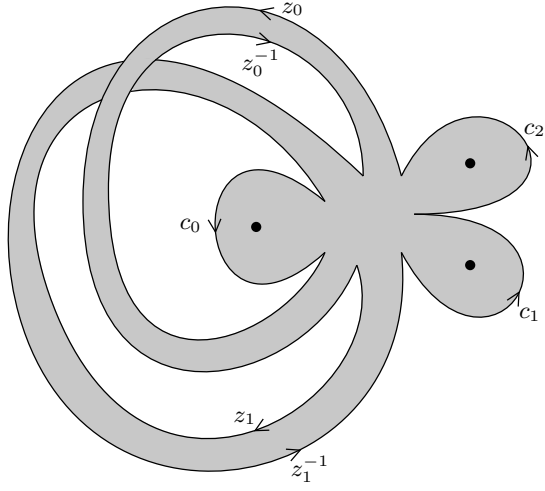

\begin{example}\label{ex:fundamental_domain}
Figure~\ref{fig:fundamental domain} shows the fundamental 
domain $R$ for a surface of genus one with one boundary component and 
three orbifold points.  Here all the orders $o_j$ are $4$.
To illustrate 
how the generators act, we've numbered the images of the point marked $0$ 
under the successive subwords of the boundary word 
$c_0z_0^{-1}z_1^{-1}c_1c_2z_0z_1$.  Recall $\Gamma$ acts on $\H^2$ on the 
left, so the successive subwords are suffixes of the boundary word.  
The induced cyclic order $O_\SS$ can be read off counterclockwise 
from the boundary of $R$, and it is $[c_1, c_2, z_0, z_1^{-1}, c_0, z_0^{-1}, z_1]$.
Figure~\ref{fig:topological_type} shows what the orbifold looks like, 
topologically.  Note it is easy to read off the boundary 
word and the cyclic order from Figure~\ref{fig:topological_type}.
\end{example}

\subsection{Core graphs of realizations}

Let $\SS$ be a realization of $\Gamma$ with generating set 
$\{c_0,\ldots, c_{J-1},z_0,\ldots, z_{I-1}\}$
and fundamental domain $R$, as above.  We will use the choice of 
generators and fundamental domain to define a graph on the orbifold $\SS$, 
as follows.  Recall that as part of the realization, we have points on the boundary 
of $R$ labeled by the elements of
$S = \{z_0, z_0^{-1}, \ldots, z_{I-1}, z_{I-1}^{-1}, c_0, \ldots, c_{J-1}\}$.
These points induce the cyclic order $O_\SS$ on $S$.

Let $p$ be a point in the interior of $R$ (a more central point makes a nicer picture, 
but it doesn't matter where it is).  Construct a directed graph 
$G'_\SS$ on $R$ with vertex set $\{p\}\cup S$ with edges as follows:
for each vertex $c_j$ in $S$, there is an edge from $p$ to $c_j$.  
For each vertex $z_i$, there is an edge from $p$ to $z_i$, and 
for each vertex $z_i^{-1}$, there is an edge from $z_i^{-1}$ to $p$.  
These edges can be made all embedded and disjoint in $R$: $R$ is 
topologically a disk, so we 
can clearly connect an interior point to arbitrary points on the boundary 
with a series of disjoint, embedded arcs. For example, we 
could make them geodesic arcs.

Now let $G_\SS$ be the quotient of $G'_\SS$, which is a graph in $\SS$.
The graph $G_\SS$ is the \emph{core graph} of the orbifold $\SS$ realizing $\Gamma$.
We now describe $G_\SS$ and name its parts so we can refer to them later.
The quotient map from $R$ to $\SS$ is an  embedding away from the boundary, so 
to know what $G_\SS$ is, it suffices to consider what happens to the 
vertices $z_i^{\pm 1}$ and $c_j$.  The $c_j$ vertices in $G'_\SS$ 
are each sent to one of the cone points.  The pair of vertices $z_i^{\pm 1}$ 
in $G'_\SS$ are identified into a single vertex in $G_\SS$, 
which we will denote by $z_i$.  So the vertex set of the core graph $G_\SS$ 
is $\{p,z_0,\ldots,z_{I-1},c_0,\ldots,c_{J-1}\}$.  For each 
$i$, there is an edge from $p$ to $z_i$ and from $z_i$ to $p$, and 
for each $j$, there is an edge from $p$ to $c_j$.
Figure~\ref{fig:core_graph} shows the core graph for the 
orbifold given in Example~\ref{ex:fundamental_domain}.

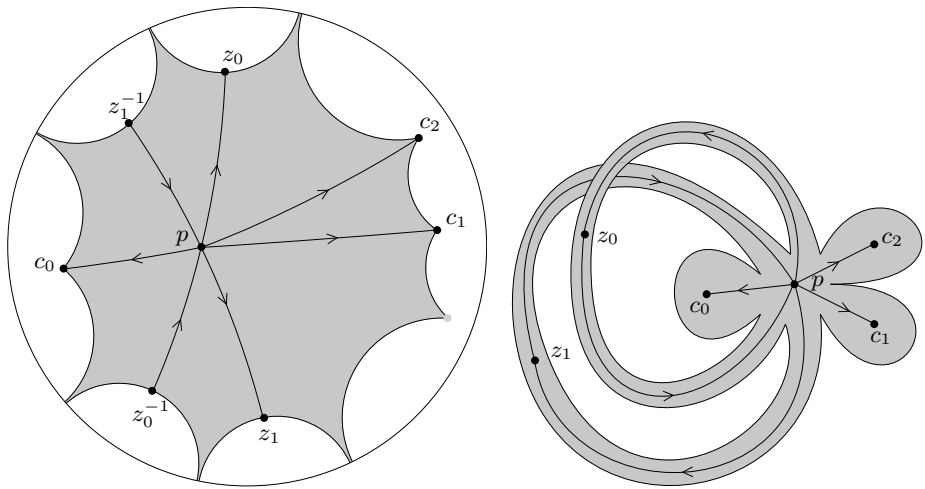
\begin{figure}[[ht]
\begin{center}
\labellist
\small\hair 2pt
 \pinlabel {$z_0$} at 181 344
 \pinlabel {$z_1$} at 210 40
 \pinlabel {$z_0^{-1}$} at 115 58
 \pinlabel {$z_1^{-1}$} at 96 305
 \pinlabel {$c_0$} at 30 178
 \pinlabel {$c_1$} at 360 213
 \pinlabel {$c_2$} at 339 290
 \pinlabel {$p$} at 140 199
\endlabellist
\includegraphics[scale=0.48]{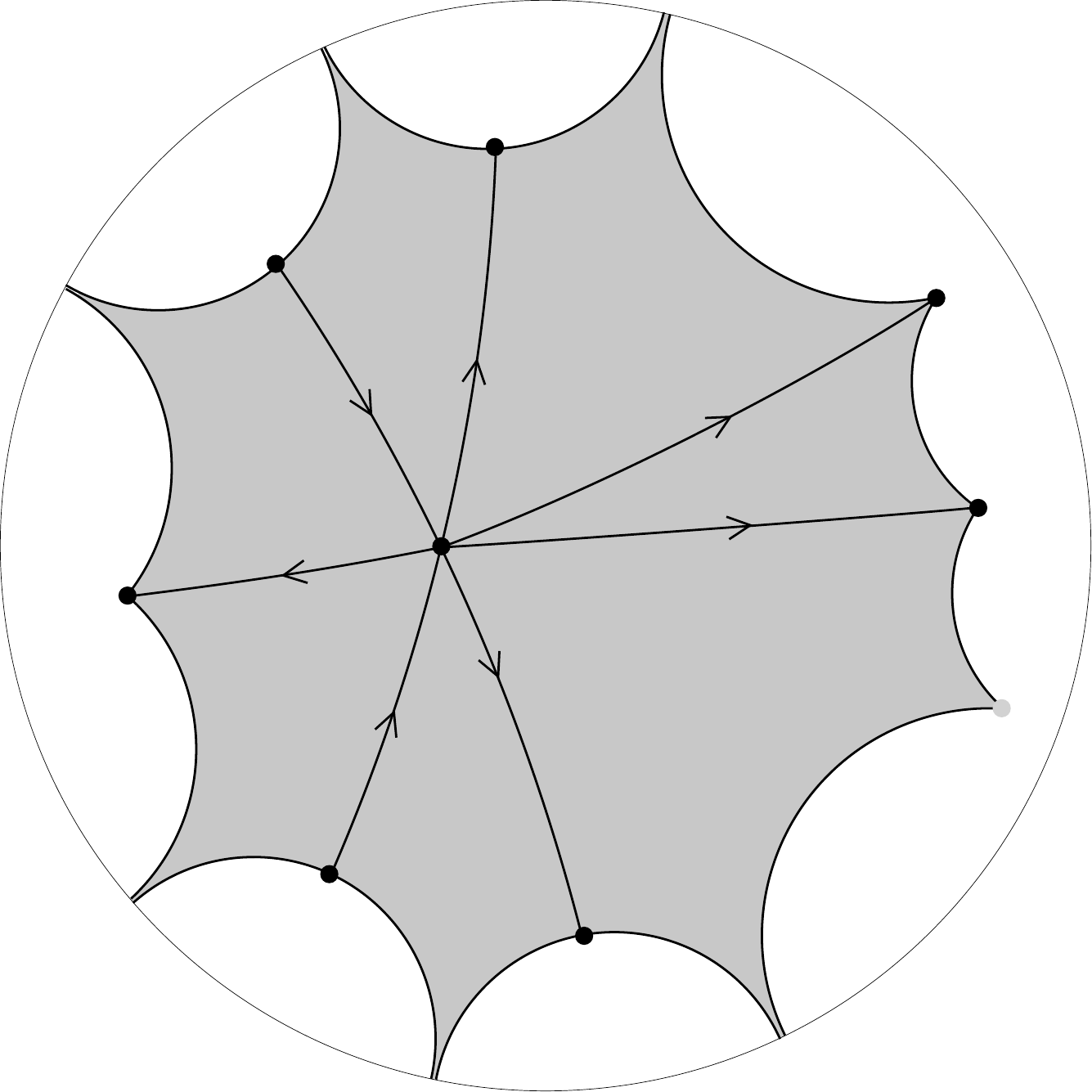}
\hspace{1mm}
\labellist
\small\hair 2pt
 \pinlabel {$c_2$} at 305 198
 \pinlabel {$c_0$} at 151 143
 \pinlabel {$c_1$} at 296 119
 \pinlabel {$z_0$} at 77 199
 \pinlabel {$z_1$} at 40 107
 \pinlabel {$p$} at 245 163
\endlabellist
\includegraphics[scale=0.48]{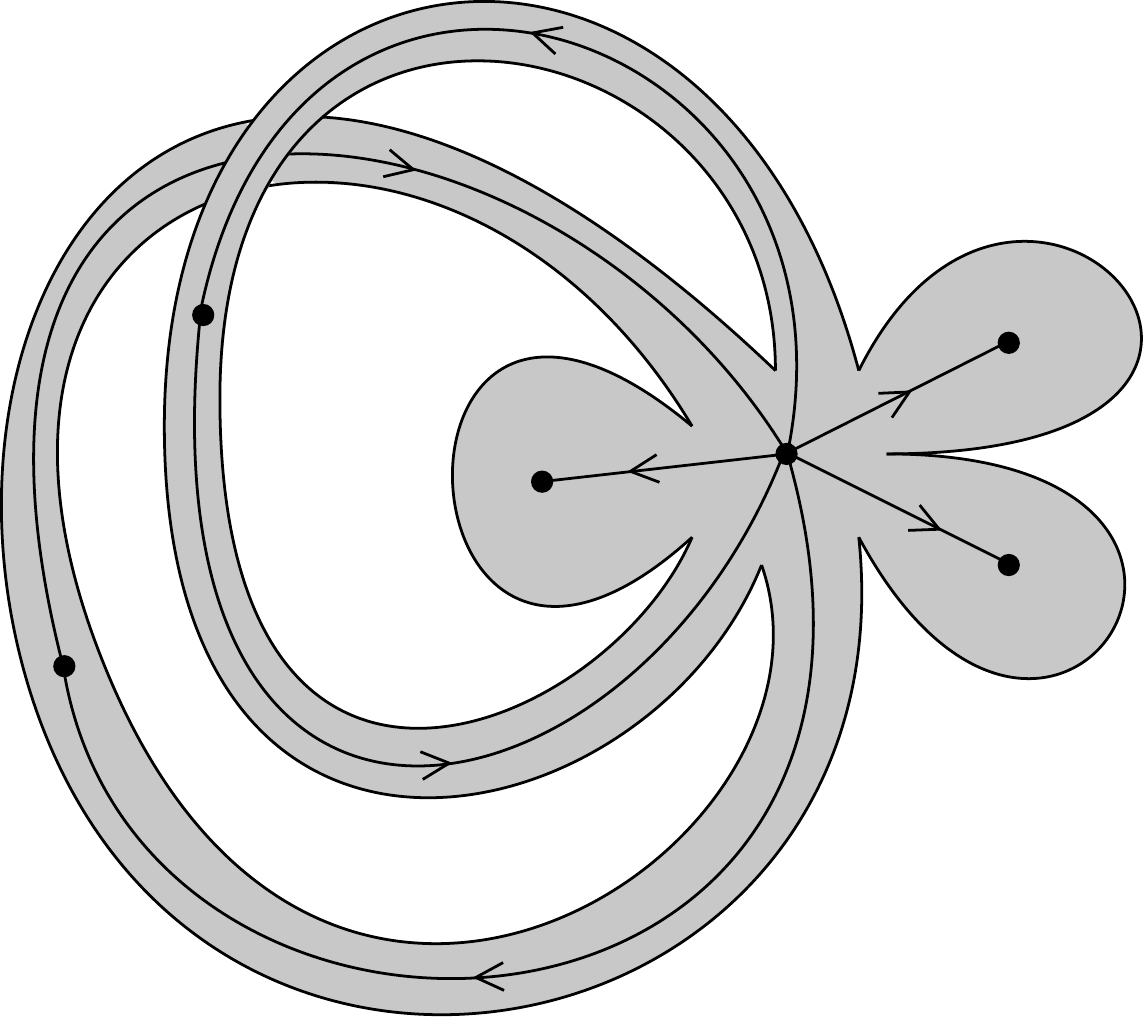}
\caption[The core graph of an orbifold]{The core graph of the realization described in 
Example~\ref{ex:fundamental_domain}.  Left, the graph $G'_\SS$ is shown 
drawn in the fundamental domain $R$.  Right, the core 
graph $G_\SS$ is shown on the orbifold, drawn as in Figure~\ref{fig:topological_type}.}
\label{fig:core_graph}
\end{center}
\end{figure}

Consider the vertex $p$ in the graph $G_\SS$.  Since the graph $G_\SS$ 
comes with an embedding in $\SS$, the vertex $p$ 
has a cyclic order on the incident edges obtained by simply reading the 
directed labels in counterclockwise order around $p$ (where the incoming edge 
from $z_i$ to $p$ is read as $z_i^{-1}$).  Note that 
these labels are exactly $S$, and this cyclic order is exactly $O_\SS$.

\subsection{Covering trees of core graphs}

The core graph of the hyperbolic orbifold $\SS$ realizing $\Gamma$ 
is a graph $G_\SS$ embedded in $\SS$.  
The preimage $\widetilde{G}_\SS$ of $G_\SS$ in 
the universal cover $\widetilde{\SS} = \H^2$ is a graph in $\H^2$.  
We call this graph the \emph{covering tree} of $G_\SS$.  We will 
verify momentarily that it is, in fact, a tree.
In $G_\SS$, each vertex $c_j$ has just a single incoming edge.  
In the covering tree, the preimages of the vertex $c_j$ 
have $o_j$ incoming edges, where recall $o_j$ is the order of the 
generator $c_j$.  This is quite natural, since 
the covering map $\H^2 \to \SS$ branches at the preimages of the 
cone points.  See Figure~\ref{fig:core_graph_covering_tree}.

\begin{figure}[[ht]
\begin{center}
\includegraphics[scale=0.6]{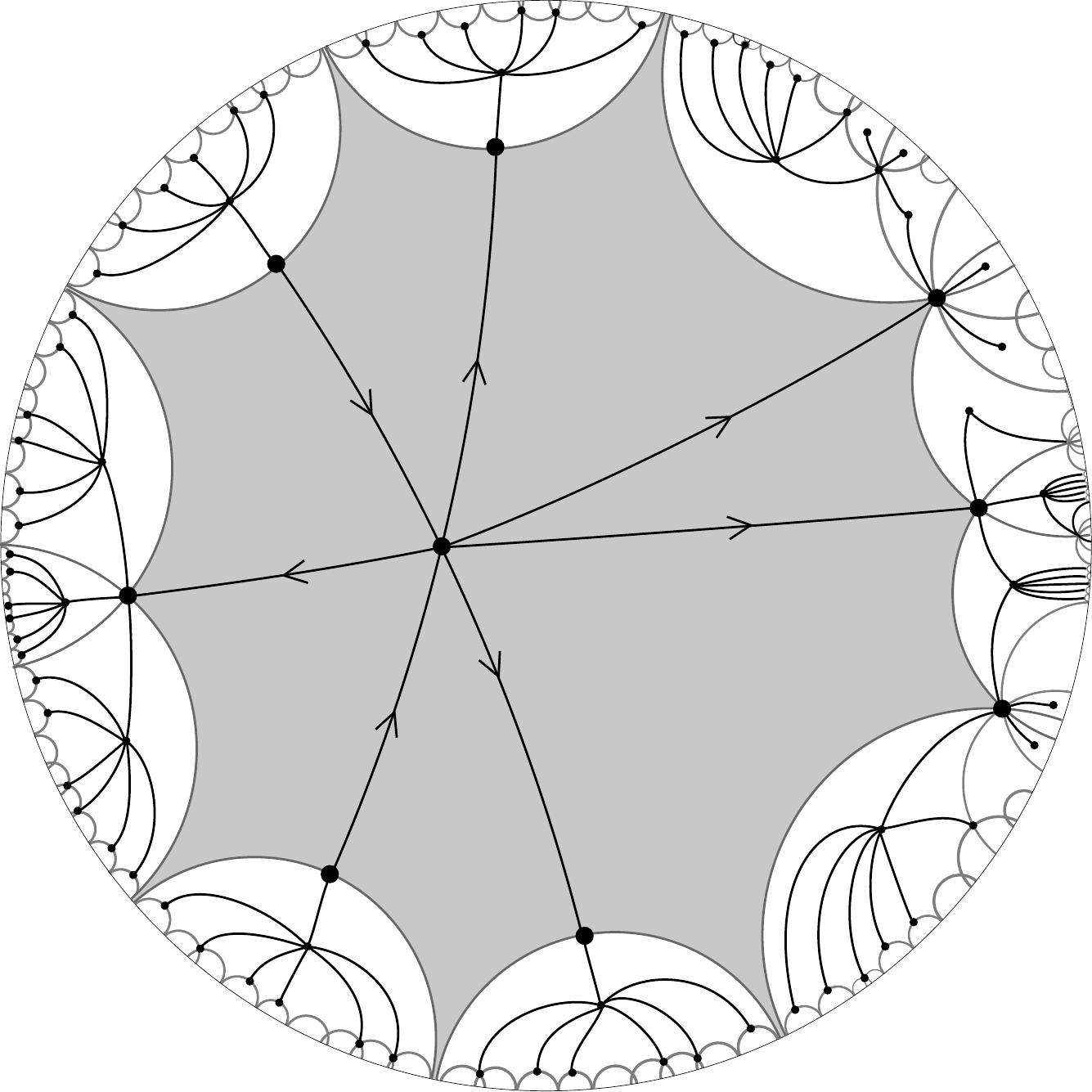}
\caption[The covering tree of the core graph of an orbifold]{The covering tree of 
the core graph in Figure~\ref{fig:core_graph}.}
\label{fig:core_graph_covering_tree}
\end{center}
\end{figure}

Since $G_\SS$ is embedded in $\SS$, the preimage $\widetilde{G}_\SS$ 
is embedded in $\H^2$.  The fact that $\widetilde{G}_\SS$ is a 
tree is quite straightforward to see intuitively, since 
$\H^2$ is the universal cover of $\SS$, and $G_\SS$ carries part of the 
fundamental group, but we go through it carefully.
To see that it is a tree, suppose that 
we have a loop $\gamma$ in $\widetilde{G}_\SS$.  Now, $\gamma$ must 
pass through a preimage of the vertex $p$ since the vertices $z_i$ and
$c_j$ are connected only to $p$ in $G_\SS$.  
Reading the vertex labels around $\gamma$ 
gives a word in $\Gamma$ taking a preimage of $p$ to itself 
(where each time we pass through a preimage of $c_j$, 
we must choose the appropriate power of $c_j$ to obtain the desired 
angle, and each time we pass through $z_i$, we record $z_i$ or $z_i^{-1}$, 
depending on whether we crossed the edges adjacent to $z_i$ 
respecting the direction).  Any word in $\Gamma$ taking a preimage of $p$ to 
itself is trivial, so it must be (a conjugate of) the word $c_j^{no_j}$ 
for some $n$.  But this word produces a trivial path, so $\gamma$  
is a trivial loop, and we see that $\widetilde{G}_\SS$ must be a tree.

\section{Cyclic fatgraphs and immersions in orbifolds}
\label{sec:immersion}

\subsection{Cyclic fatgraphs}

Our proofs will build surface immersions using \emph{cyclic fatgraphs over $\Gamma$}, 
which are combinatorialized surface maps into $\SS$.  A cyclic fatgraph over $\Gamma$ 
is a surface which is built out of \emph{pieces}, which are \emph{rectangles}, 
\emph{polygons}, and \emph{group polygons}, and where the pieces are 
glued along \emph{edges}.  We now define all these terms.  

A \emph{rectangle} is a $2$-cell whose boundary is an oriented 
simplicial loop with $4$ $1$-simplices.  
We think of a rectangle as a rectangular strip.  It is labeled on one 
long side by an infinite order generator $z_i$ and 
on the other by its inverse $z_i^{-1}$.  The notation 
for such a rectangle is $r(z_i)$.  The short $1$-simplices 
of the rectangle $r(z_i)$ are \emph{rectangle edges}, 
and a rectangle edge is denoted by $re(z_i)$ or $re(z_i^{-1})$ 
depending on which long labeled side comes \emph{after} 
the rectangle edge.  See Figure~\ref{fig:rectangle}.

\begin{figure}[ht]
\labellist
\small\hair 2pt
 \pinlabel {$z_i$} at 81 -7
 \pinlabel {$re(z_i^{-1})$} at 200 47
 \pinlabel {$z_i^{-1}$} at 72 93
 \pinlabel {$re(z_i)$} at -36 47
\endlabellist
\includegraphics[scale=0.6]{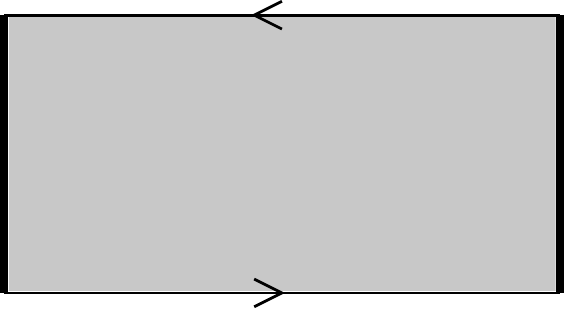}
\caption{A rectangle, with two labeled sides and two edges (bold).}
\label{fig:rectangle}
\end{figure}
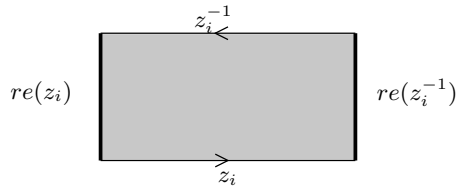

A \emph{group polygon} is a $2$-cell whose boundary is an oriented 
simplicial loop with simplices alternating between 
labeled sides and \emph{group polygon edges}.  The labeled sides are all labeled 
by the same finite order generator, and there must be exactly 
as many labeled sides as the order $o_j$ of $c_j$.  
Every group polygon edge in such a group polygon is denoted by $ge(c_j)$.
The notation for this group polygon is $g(c_j)$.
See Figure~\ref{fig:group_polygon}.

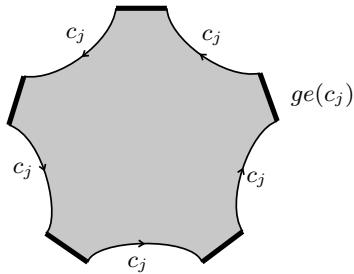
\begin{figure}[ht]
\labellist
\small\hair 2pt
 \pinlabel {$c_j$} at 118 40
 \pinlabel {$ge(c_j)$} at 149 80
 \pinlabel {$c_j$} at 97 107
 \pinlabel {$c_j$} at 33 108
 \pinlabel {$c_j$} at 8 44
 \pinlabel {$c_j$} at 62 0
\endlabellist
\includegraphics[scale=0.8]{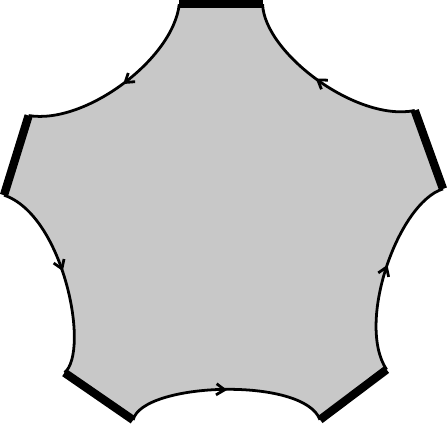}
\caption{A group polygon in the case that $o_j=5$.}
\label{fig:group_polygon}
\end{figure}

A \emph{polygon} is a $2$-cell 
whose boundary is an oriented simplicial complex 
whose simplicies are all \emph{polygon edges}.  
A polygon edge can be one of $pe(c_j)$, $pe(z_i)$, 
or $pe(z_i^{-1})$.  There is a restriction 
that a polygon must be locally reduced, which means that 
$pe(z_i)$ cannot immediately follow $pe(z_i)$, and similarly 
for the inverses.
A nondegenerate polygon can have two or more edges.  We will often refer 
to polygons with the name appropriate to their number of edges, 
for example bigon, triangle, square, etc.
See Figure~\ref{fig:polygon}.  For technical reasons, it 
is convenient to allow polygons with a single edge (a monogon).  
Such polygons may only have edges of the form $pe(c_j)$ 
(a finite-order generator).  Monogons are needed to allow 
repeated copies of $c_j$ to appear on the boundary of 
a cyclic fatgraph.  Figure~\ref{fig:pinching} contains an example.

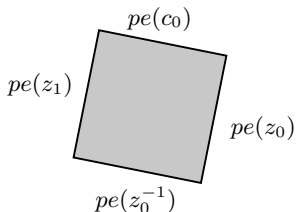
\begin{figure}[ht]
\labellist
\small\hair 2pt
 \pinlabel {$pe(z_0)$} at 240 70
 \pinlabel {$pe(c_0)$} at 110 208
 \pinlabel {$pe(z_1)$} at -40 124
 \pinlabel {$pe(z_0^{-1})$} at 80 -20
\endlabellist
\includegraphics[scale=0.3]{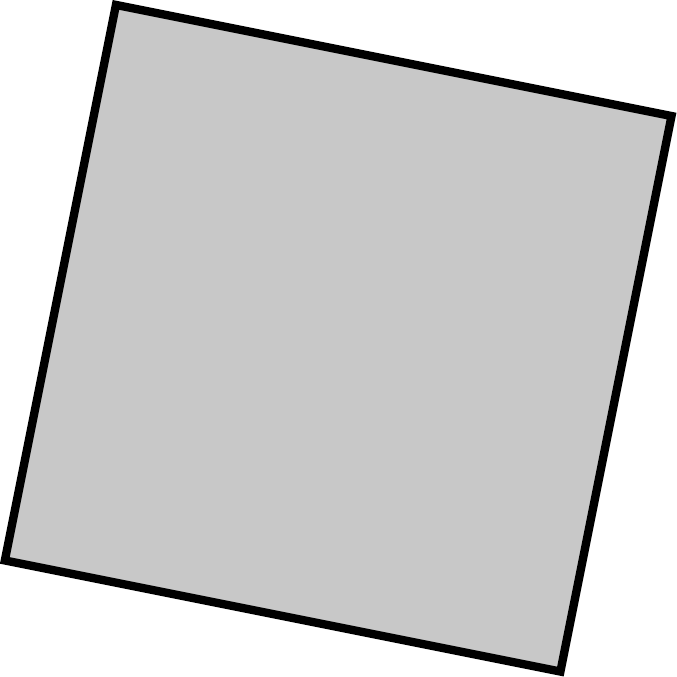}
\caption{A polygon.  All the sides are edges.  This square 
appears in the cyclic fatgraph shown in Figure~\ref{fig:pinching}.}
\label{fig:polygon}
\end{figure}

When dealing with these pieces, we will often refer to finite-order 
or infinite-order edges, meaning rectangle, group polygon, 
and polygon edges, as appropriate.

A \emph{cyclic fatgraph over $\Gamma$} is a surface with 
a simplicial structure such that every $2$-cell has the structure 
of a rectangle, polygon, or group polygon.
If a $1$-simplex is the boundary of two $2$-cells, 
then one of the $2$-cells must be a polygon and the other 
must be a rectangle or group polygon, and the simplex of 
intersection must be an edge in both, and the labels must match, e.g. 
$pe(z_i)$ is glued to $re(z_i)$.
That is, a 
cyclic fatgraph over $\Gamma$ is a surface built out of 
rectangles and group polygons by gluing them
together around polygons along edges.  See Figure~\ref{fig:cyclic_fatgraph}.
In our drawings, including Figure~\ref{fig:cyclic_fatgraph}, 
note that where the rectangles and group polygons appear to 
attach directly to each other, there is technically a bigon 
(polygon with two sides) joining them.  This technicality 
is useful to avoid special cases in the definition and 
for some definitions to follow.

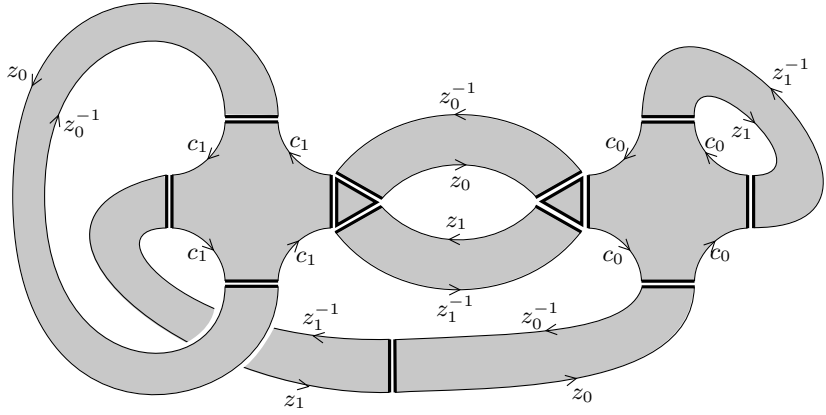
\begin{figure}[ht]
\labellist
\small\hair 2pt
 \pinlabel {$z_0$} at 430 0
 \pinlabel {$z_0^{-1}$} at 398 60
 \pinlabel {$c_0$} at 534 106
 \pinlabel {$c_0$} at 530 189
 \pinlabel {$c_0$} at 455 189
 \pinlabel {$c_0$} at 453 104
 \pinlabel {$z_1^{-1}$} at 587 244
 \pinlabel {$z_1$} at 550 198
 \pinlabel {$z_0^{-1}$} at 336 226
 \pinlabel {$z_0$} at 338 160
 \pinlabel {$z_1$} at 335 130
 \pinlabel {$z_1^{-1}$} at 334 65
 \pinlabel {$c_1$} at 217 190
 \pinlabel {$c_1$} at 140 188
 \pinlabel {$c_1$} at 140 106
 \pinlabel {$c_1$} at 222 103
 \pinlabel {$z_0$} at 4 243
 \pinlabel {$z_0^{-1}$} at 53 205
 \pinlabel {$z_1$} at 213 -5
 \pinlabel {$z_1^{-1}$} at 233 59
\endlabellist
\includegraphics[scale=0.5]{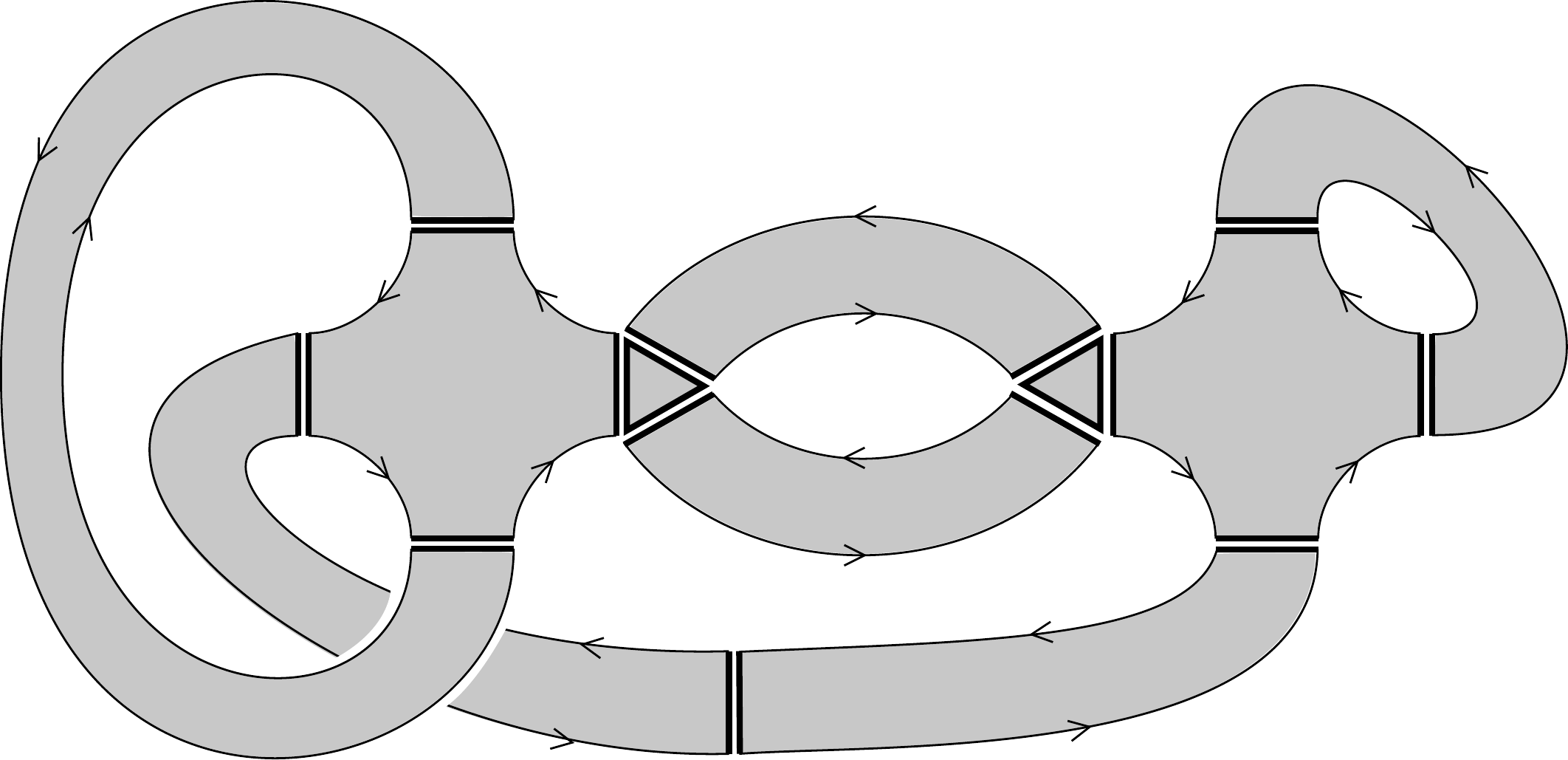}
\caption{A cyclic fatgraph over the group 
$\Gamma$ from Example~\ref{ex:fundamental_domain}.  Recall 
that $o_j=4$ for all $j$, so each group polygon has four sides.  
This fatgraph has boundary
$z_0c_0z_1^{-1}c_0z_0^{-1}c_1z_0c_1z_1^{-1}c_0z_0^{-1}z_1^{-1}c_1z_0^{-1}c_1z_1 
+ z_0z_1 + z_1c_0$.  There are seven bigons and two triangles.  }
\label{fig:cyclic_fatgraph}
\end{figure}

\subsection{Spines of cyclic fatgraphs}

Given a cyclic fatgraph $Y$, we define the \emph{spine} $G_Y$ 
of $Y$, which is a directed graph on $Y$, to be the graph 
dual to the cellulation of $Y$ by the polygons, rectangles, and 
group polygons in $Y$.  Since there is a vertex for each piece, 
we call the vertices polygon, rectangle, or group polygon vertices 
as appropriate.  Orient the edges of $G_Y$ so that 
every edge between a polygon vertex and a group polygon vertex is directed 
away from the polygon vertex.  Orient the two edges incident to a 
rectangle vertex so that their orientation agrees with the orientation on the 
side of the rectangle labeled by $z_i$ (and against the orientation 
on the side labeled by $z_i^{-1}$).  
Note that $Y$ deformation retracts to its spine.  
See Figure~\ref{fig:spine}.

\begin{figure}[ht]
\labellist
\small\hair 2pt                                                                                                                                                
 \pinlabel {$c_1$} at 136 61                                                                                                                                 
 \pinlabel {$c_1$} at 133 96                                                                                                                                 
 \pinlabel {$c_1$} at 101 95                                                                                                                                  
 \pinlabel {$c_1$} at 101 60                                                                                                                                 
 \pinlabel {$c_0$} at 63 60                                                                                                                                  
 \pinlabel {$c_0$} at 62 96                                                                                                                                  
 \pinlabel {$c_0$} at 31 96                                                                                                                                  
 \pinlabel {$c_0$} at 28 61                                                                                                                                  
 \pinlabel {$z_1^{-1}$} at 115 22                                                                                                                                 
 \pinlabel {$z_1$} at 98 46                                                                                                                                  
 \pinlabel {$z_0^{-1}$} at 163 118                                                                                                                                
 \pinlabel {$z_0$} at 143 100                                                                                                                                
 \pinlabel {$z_0$} at 96 103                                                                                                                                 
 \pinlabel {$z_0^{-1}$} at 82 126                                                                                                                                 
 \pinlabel {$z_1$} at 22 101                                                                                                                                 
 \pinlabel {$z_1^{-1}$} at 8 121                                                                                                                                  
 \pinlabel {$z_0$} at 65 -2                                                                                                                                   
 \pinlabel {$z_0^{-1}$} at 51 21                                                                                                                                  
\endlabellist  
\includegraphics[scale=1.2]{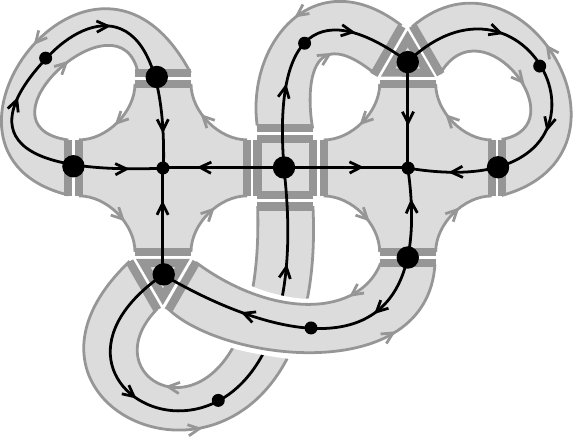}
\caption{The spine of a cyclic fatgraph.  The polygon vertices are 
drawn larger to differentiate them from the rectangle and group polygons 
vertices.  Note there are four 
polygon vertices in the (invisible) bigons.}
\label{fig:spine}
\end{figure}
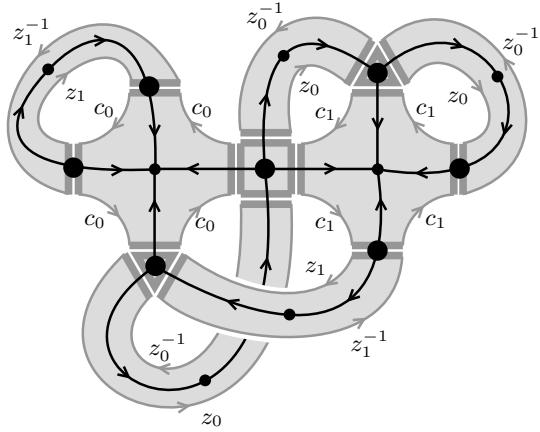

\subsection{Immersed surfaces in orbifolds}

Let $\SS$ be an orbifold realizing the group $\Gamma$ 
with core graph $G_\SS$, and suppose that we have a cyclic fatgraph $Y$ 
over $\Gamma$, as defined above.  There is a natural 
simplicial map from the spine $G_Y$ of $Y$ to the core graph $G_\SS$ 
defined by sending the polygon, rectangle, and group polygon 
vertices of $G_Y$ to the $p$, $z_i$, and $c_j$ vertices of $G_\SS$, 
and by mapping the edges so as to preserve orientation.
Let $f_Y:Y \to \SS$ be the map deformation retracting $Y$ to its spine, and 
then mapping $G_Y$ to $G_\SS$ as above.  We call this map the \emph{fatgraph map}
induced by $Y$.  Note that $f_Y$ is a map of a surface with boundary 
into the orbifold $\SS$.  Though we will not need it, 
we remark that every map of a surface factors through a fatgraph map:

\begin{lemma}[Lemma~4.4 in \cite{W_scylla}]
\label{lem:factor_through_fatgraph}
After compression and homotopy, 
every surface map $f:S \to \SS$ factors as $f_Y \circ i$, 
where $i:S \to Y$ is a homeomorphism between $S$ and 
a cyclic fatgraph over $\Gamma$, and $f_Y$ is the 
fatgraph map.
\end{lemma}

In~\cite{W_scylla}, the lemma is not stated exactly in this 
way, but it follows from the proof.  In the case of 
free groups, this lemma is due to Culler \cite{Culler}.

The main result of this section is that there is a local 
combinatorial certificate that the map $f_Y$ is 
homotopic to an immersion with geodesic boundary.  We now 
describe this certificate.

Recall that $S$ is the set of generators
$\{z_0, z_0^{-1}, \ldots, z_{I-1}, z_{I-1}^{-1}, c_0, \ldots, c_{J-1}\}$ of $\Gamma$, 
and $O_\SS$ is the cyclic order on $S$ determined by the realization $\SS$.  
Let us be given a polygon $P$ in a 
cyclic fatgraph over $\Gamma$.  The edges of 
$P$ have an intrinsic cyclic order, so the set 
of labels on the edges of $P$ is a cyclically ordered multiset. 
The cyclically ordered multiset 
of labels of the edges of $P$ will be denoted by $\partial P$.
For example, if the edges of $P$ are 
$[pe(z_0^{-1}), pe(c_0), pe(c_1), pe(z_1^{-1}), pe(c_0)]$, 
then the set of labels is the cyclically ordered multiset 
$\partial P = [z_0^{-1}, c_0, c_1, z_1^{-1}, c_0]$. 
We call a polygon \emph{small} if $\partial S$ is actually a set; i.e. 
if each label appears at most once in $\partial P$.  Notice that 
if $P$ is small, then $\partial P \subseteq S$, so there are 
two cyclic orders on $\partial P$: its intrinsic cyclic order 
and the cyclic order given by $O_\SS$.

There is a special polygon, the \emph{standard $\SS$ polygon}, 
which is the polygon such that $\partial P = O_\SS$, i.e. 
every outgoing edge appears exactly once, and in the cyclic order 
$O_\SS$.  The standard $\SS$ polygon is the largest polygon 
whose boundary is compatibly ordered with $O_\SS$.

\begin{proposition}\label{prop:immersion}
Let $Y$ be a cyclic fatgraph over $\Gamma$ with induced surface map 
$f_Y:Y \to \SS$.  Suppose that every boundary 
component of $Y$ is realized in $\SS$ by a geodesic loop (is not 
finite order or parabolic).  If every polygon $P$ in $Y$ has the 
property that $P$ is small and the cyclic order on $\partial P$ 
is compatible with $O_\SS$, then $f_Y$ is homotopic to 
an immersion with geodesic boundary.
\end{proposition}
\begin{remark}
Cyclic orders are useful for many geometric things; 
see \cite{C_foliations}, Chapter~2 and \cite{C_scl}, Section~4.2.5.
Proposition~\ref{prop:immersion} is essentially 
a generalization of the ideas in \cite{C_scl}, Section~4.2.5 to orbifolds.
\end{remark}

\begin{remark}
Consider Figure~\ref{fig:spine}, the spine of a fatgraph and
Figure~\ref{fig:core_graph_covering_tree}, the covering tree of the 
core graph of a realization.  The map $f_Y$ retracts the fatgraph to 
the spine and then sends the spine inside the core graph.  In the universal 
cover, then, the covering map $\widetilde{f_Y}$ sends the universal cover of the spine 
inside the covering tree of the core graph.  The hypothesis 
of Proposition~\ref{prop:immersion} makes sure that this covering 
map preserves the cyclic orders of the edges around each vertex and is 
thus an embedding.
Since the universal cover of the fatgraph retracts to the spine and the spine is embedded, 
Proposition~\ref{prop:immersion} is quite natural.  The 
proof is a formalization of this.
\end{remark}

\begin{proof}[Proof of Proposition~\ref{prop:immersion}]
In order to prove that $f_Y$ is homotopic to an immersion, we must show that 
the lift $\widetilde{f_Y} : \widetilde{Y} \to \widetilde{\SS}$ is 
homotopic to an immersion, where the homotopy of $\widetilde{f_Y}$ must 
be equivariant with respect to $\pi_1(Y)$.

Write $\widetilde{f_Y} = g\circ h$, where $h$ is the deformation 
$\widetilde{Y} \to \widetilde{G_Y}$, and $g$ is the simplicial 
graph map $\widetilde{G_Y} \to \widetilde{G_\SS}$.  
Let us consider what happens to the stars of the vertices 
under the graph map $g$.  There are three kinds of vertices in $\widetilde{G_Y}$, 
covering rectangle, group polygon, and polygon vertices.  
The stars of rectangle vertices are $2$-valent and map to the $2$-valent stars of the 
$z_i$ vertices in $\widetilde{G_\SS}$.  A group polygon 
vertex $v$ corresponding to $c_j$ has valence $o_j$ and maps to a 
vertex $w$ covering a torsion vertex $c_j$
of the core graph.  The vertices $v$ and $w$ each have $o_j$ incoming edges, and 
because the star of $v$ covers the $1$-valent star of the projection of $w$ with 
degree $o_j$, the $o_j$-valent star of $v$ is identified with the $o_j$-valent 
star of $w$.  This uses the fact that the graphs are embedded in the surfaces 
$Y$ and $\SS$; there is no angle structure on an abstract graph, 
but there is for the graphs $G_Y$ and $G_\SS$, so we know 
how the cone points are covered.

Finally, consider the star of a polygon vertex $v$ in $\widetilde{G_Y}$, 
which maps to a vertex $w$ covering $p$ in $G_\SS$.  
By assumption, the multiset of incident edge labels at $v$ is a subset of 
the incident edges at $w$ (the polygon is small), and they are 
compatibly cyclically ordered (the cyclic order on the polygon is 
compatible with $O_\SS$).  Therefore, the star of $v$ is embedded 
in the star of $w$.

We conclude that the map $g$ embeds the tree $\widetilde{G_Y}$ 
inside the tree $\widetilde{G_\SS}$ in a way that 
preserves the cyclic order on every vertex.  
If the reader is familiar with pleated surfaces (see \cite{Thurston}), 
it is enough now to note that this fact about graphs implies that the 
pleated surface representative of the map $f_Y$ has only positive 
simplices and is therefore an immersion.  If not, we explain.  
Give $Y$ a hyperbolic structure with geodesic boundary and 
decompose it into ideal triangles (which will necessarily be spun 
around some closed geodesics in $Y$).  These ideal triangles 
lift to ideal triangles in the universal cover $\widetilde{Y}$.
Because $h$ is a deformation retraction to a tree, 
the image of an ideal triangle $T$ under $h$ is an infinite tripod $h(T)$, 
and because $g$ embeds the tree $\widetilde{G_Y}$ inside 
$\widetilde{G_\SS}$, 
the image of $T$ under $g\circ h = \widetilde{f_Y}$ 
is an infinite tripod $g(h(T))$.   Now, $g$ sends the three ideal 
points at the ends of $h(T)$ to the three ideal points at the ends of $g(h(T))$.  
And because $g$ preserves the cyclic order on every vertex, the 
cyclic orders on these triples of points are the same.  

Therefore, the image of $T$ under $\widetilde{f_Y}$ is an infinite 
tripod whose ends have the same cyclic order as the ends of $T$.  There is a 
geodesic ideal triangle $T'$ in $\widetilde{\SS}$ 
with the same ends as $\widetilde{f_Y}(T)$, and we can homotope $\widetilde{f_Y}$ 
on $T$ to map $T$ to $T'$.  Because the order on the ends is preserved, 
this map is orientation-preserving.

Do this homotopy on lifts of each ideal triangle in $Y$, and extend the 
homotopy equivariantly over $\widetilde{Y}$.  The result is an 
equivariant homotopy of $\widetilde{f_Y}$ to a map which takes 
ideal triangles to ideal triangles in an orientation-preserving 
way; that is, it is an immersion taking geodesic boundary to 
geodesic boundary.

\end{proof}

\begin{example}
While Proposition~\ref{prop:immersion} may seem technical, it 
is straightforward to apply in practice.  
Consider the orbifold $\SS$ from Example~\ref{ex:fundamental_domain} 
with fundamental group $\Gamma$.
Figure~\ref{fig:spine} shows a cyclic fatgraph $Y$
over $\Gamma$.  A simple check at the polygons shows that 
the cyclic orders are $[c_1, z_0, c_0, z_0^{-1}]$, $[c_1, z_0, z_0^{-1}]$, 
$[c_0, z_0, z_1^{-1}]$, $[c_1,z_0^{-1}]$, $[c_1,z_1]$, $[c_0,z_1^{-1}]$,
and $[c_0,z_1]$, which are all compatible with $O_\SS$.  Thus, the 
map $f_Y$ can be straightened to 
an immersion with geodesic boundary, and in particular, 
the boundary loops bound an immersed surface.
\end{example}

\subsection{Building immersed surfaces}
\label{sec:building_immersed_surfaces}

Given a fatgraph, it is easy to check using 
Proposition~\ref{prop:immersion} whether the induced surface 
map is homotopic to an immersion with geodesic boundary.  
Our goal in this paper is to build fatgraphs 
(1) which are homotopic to immersions with geodesic boundary and 
(2) which have some given word in the generators as a boundary.
In this section, we show that relaxing either of this 
conditions makes the problem trivial.  This section is 
mainly background and introduction to the methods we will 
use later.

\subsubsection{Cyclic fatgraphs which satisfy Proposition~\ref{prop:immersion}}
\label{sec:construction_satisfy_immersion}

Constructing fatgraphs which satisfy the hypotheses of 
Proposition~\ref{prop:immersion} is quite straightforward.  The proposition
requires that the fatgraph be built using only small polygons 
whose intrinsic cyclic order on edges is compatible with $O_\SS$.
So if we simply enumerate all possible polygons satisfying this hypothesis, and 
all rectangles and group polygons, then we can take any subset of these pieces 
such that each edge occurs the same number of times in polygons 
as it does in rectangles and group polygons and then glue these pieces 
together arbitrarily.

\subsubsection{Cyclic fatgraphs over $\Gamma$ with given boundary $w$}
\label{sec:construction_boundary}

The boundary $\gamma$ of $Y$ has a simplicial structure, 
and every (oriented) $1$-simplex in $\gamma$ is labeled by a generator 
inherited from the labels on the pieces of $Y$.  
If we simply read off the labels as we follow $\gamma$, that tells us 
the image word $f_Y(\gamma)$ in $\Gamma$.

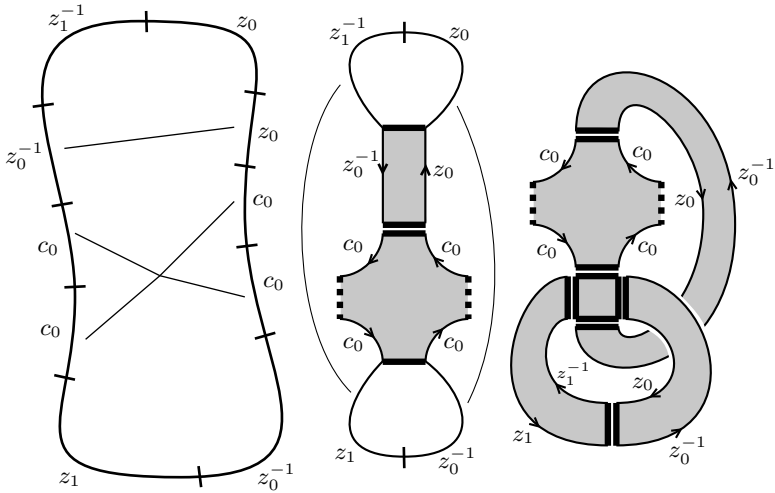
\begin{figure}[ht]
\labellist
\small\hair 2pt
 \pinlabel {$c_0$} at 91 77
 \pinlabel {$c_0$} at 87 108
 \pinlabel {$z_0$} at 89 133
 \pinlabel {$z_0$} at 81 175
 \pinlabel {$z_1^{-1}$} at 13 177
 \pinlabel {$z_0^{-1}$} at -3 125
 \pinlabel {$c_0$} at 6 90
 \pinlabel {$c_0$} at 7 59
 \pinlabel {$z_1$} at 14 3
 \pinlabel {$z_0^{-1}$} at 92 4
 
 \pinlabel {$c_0$} at 159 54
 \pinlabel {$c_0$} at 158 90
 \pinlabel {$z_0$} at 155 119
 \pinlabel {$z_0$} at 161 171
 \pinlabel {$z_1^{-1}$} at 119 172
 \pinlabel {$z_0^{-1}$} at 124 122
 \pinlabel {$c_0$} at 121 91
 \pinlabel {$c_0$} at 121 54
 \pinlabel {$z_1$} at 118 12
 \pinlabel {$z_0^{-1}$} at 160 10
 
 \pinlabel {$c_0$} at 232 90
 \pinlabel {$c_0$} at 230 126
 \pinlabel {$z_0$} at 246 108
 \pinlabel {$z_0$} at 230 39
 \tiny
 \pinlabel {$z_1^{-1}$} at 204 43
 \small
 \pinlabel {$z_0^{-1}$} at 273 119
 \pinlabel {$c_0$} at 195 125
 \pinlabel {$c_0$} at 194 89
 \pinlabel {$z_1$} at 185 19
 \pinlabel {$z_0^{-1}$} at 247 14
 
\endlabellist
\includegraphics{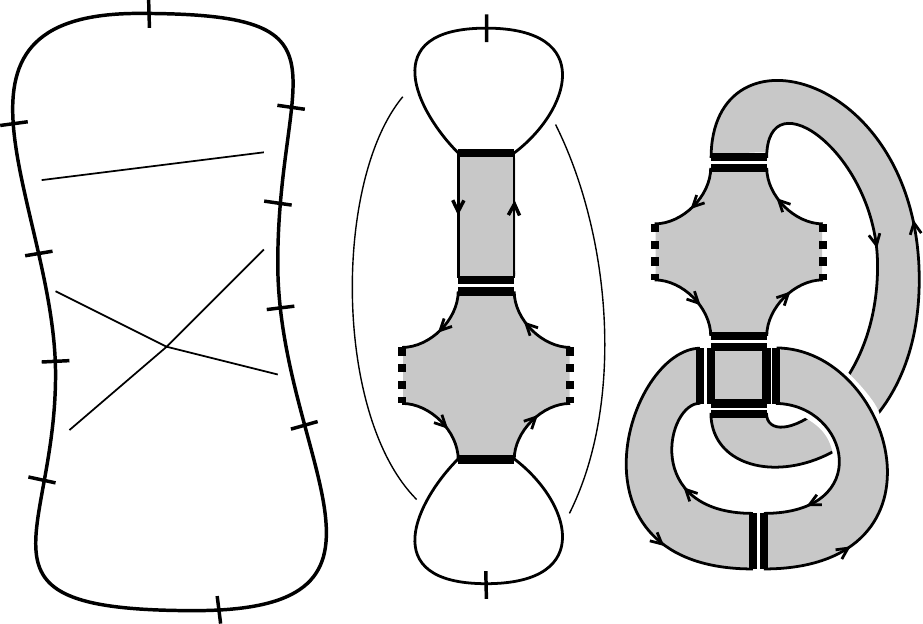}
\caption{To pinch the loop 
$c_0^2 z_0^2z_1^{-1}z_0^{-1}c_0c^2z_1z_0^{-1}$ into 
a cyclic fatgraph, we can pair up the letters arbitrarily 
into rectangles and group polygons.  The group polygon 
has has two group polygon edges which are glued to monogons, which allows 
$c_0^2$ to be part of the boundary.}
\label{fig:pinching}
\end{figure}

To build a surface map into $\SS$ with a desired boundary loop $w$, 
which must of course be homologically trivial,
we can start with an oriented simplicial circle, with each simplex 
labeled by a generator, such that the circle reads off $w$.  Because 
$w$ is homologically trivial, there are as many instances of $z_i$ as 
there are of $z_i^{-1}$, so we can pair them arbitrarily into rectangles.  
Similarly, the $c_j$ must come in groups of $o_j$, so we can group them together 
to form group polygons.  Placing polygons at the junctions of 
the group polygons and rectangles to fill in the holes, 
we produce a cyclic fatgraph $Y$ over $\Gamma$, which comes 
with the surface map $f_Y:Y\to \SS$.  By construction, the boundary 
of $Y$ maps to $w$. This construction yields a surface map bounding any 
homologically trivial word, or words, in $\Gamma$.  See Figure~\ref{fig:pinching}.

\subsubsection{Summary} We have seen that building a fatgraph 
satisfying Proposition~\ref{prop:immersion} is easy, and 
building a fatgraph with a given boundary is easy.  
However, the methods to accomplish each goal are very different, and 
note that if we glue pieces as in 
Section~\ref{sec:construction_satisfy_immersion}, it is quite 
difficult to control what the boundary is, and if we pinch a boundary 
loop as in Section~\ref{sec:construction_boundary}, it is quite 
difficult to control what polygons appear.

The proofs of 
Theorems~\ref{thm:disk_orbifolds} and~\ref{thm:genus_orbifolds} 
use Section~\ref{sec:construction_satisfy_immersion}, but done 
carefully in a way which controls the boundary.

\section{Stability}
\label{sec:stability}

\subsection{Disk orbifolds}

In this section, we prove our main result, 
which says that, under some conditions, the product of any word in $\Gamma$ 
with a sufficiently high multiple of the boundary word of $\SS$
is a loop which bounds an immersed surface with geodesic boundary in $\SS$.
First, we state and prove the 
version for orbifolds whose underlying topological space 
is a disk.

\begin{theorem}\label{thm:disk_orbifolds}
Let $\SS$ be a hyperbolic orbifold whose underlying topological space is a disk and 
which has at least three orbifold points, of orders $\{o_j\}_{j=0}^{J-1}$. 
Let $w \in \Gamma = \pi_1(\SS)$ be any hyperbolic element, and let $b \in \Gamma$ be the 
boundary loop of $\SS$.  Then there exists $N \in \N$ so that 
for all $n\ge 0$, the loop $wb^{N+ng}$ virtually 
bounds an immersed surface, where $g = \gcd(o_0-1, \ldots, o_{J-1}-1)$.
\end{theorem}

\begin{remark}
As mentioned in the introduction, the boundary of the orbifold 
is actually associated to a conjugacy class in $\Gamma$, not a specific word, 
so taking $wb^{N+ng}$ is not well-defined.  However, 
if we can prove that the theorem holds for any specific representative word 
in the conjugacy class, then the theorem holds for \emph{any} 
representative, because the effect of conjugating $b$ is actually 
just to change the word $w$.
\end{remark}

\begin{proof}

As discussed in Section~\ref{sec:building_immersed_surfaces}, 
our strategy will be to carefully piece together a fatgraph 
using only polygons allowed by Proposition~\ref{prop:immersion}
In this proof, a \emph{partial fatgraph} will be a fatgraph 
with some edges left unattached.  That is, $2$-complex 
whose cells are fatgraph pieces and whose boundary, an oriented 
simplicial $1$-complex, is allowed to contain $1$-simplices 
which are fatgraph edges.

In the first step, we build a partial fatgraph 
whose boundary contains the desired word $w$, plus 
some unglued edges.  In the second 
step, we describe how taking multiple 
copies of the partial fatgraph allows us to glue 
the unglued edges to complete the fatgraph in such a way 
the there is an integer $m$ so that the boundary is multiple loops 
of the form $(wb^m)^k$ for integers $k$.
This shows that $wb^m$ virtually bounds an immersed surface 
for some $m$.
In the final step, we describe how to vary the power $m$ to 
arrive at the result.

Let $w = w_0\cdots w_m$.
Without loss of generality, we can assume that $w$ has no cyclic 
cancellation with $b$.  
For if it does, we can 
prepend and append copies of $b$ to $w$ until 
there is no longer cyclic cancellation ($b$ has no cancellation 
with itself), and take $w$ to be this word.  These extra copies of 
$b$ are subsumed into the $N$ in the theorem.

Recall that the \emph{standard $\SS$ polygon} is the polygon $P$
such that $\partial P = O_\SS$.  Because the underlying space of 
$\SS$ is a disk, the boundary word is the product of the finite-order 
generators.  By relabeling, we may assume without 
loss of generality that $b = c_0c_1\cdots c_{J-1}$; 
this implies that the standard $\SS$ polygon has 
(cyclically ordered) boundary $[c_0,c_1,\cdots,c_{J-1}]$.  
This relabeling doesn't affect the proof, but it is simpler to think about.
Note that when we start to add infinite-order generators, 
it will \emph{not} be true that the arrangement of generators 
around the standard $\SS$ polygon has the same cyclic order 
as the boundary word.

{\bf Step 1:}
Start with a horizontal polygonal line 
oriented to the left and labeled by $w$, so the 
leftmost simplex is labeled by the final letter in $w$.
Break $w$ into runs of a single generator.  Since 
$w$ is reduced, any run of $c_j$ will have length 
less than $o_j$.  Let $W_1, \ldots, W_K$ 
denote these runs.  For each run $W_k$, which will be 
of the form $W_k = c_j^{e_k}$, build a 
group polygon which is labeled on top by $W_k$, and 
on the bottom by $o_j-e_k$ copies of $c_j$.  This group polygon
has $e_k-1$ edges on top in between letters in $W_k$; 
these will be glued to monogons so that $c_j^{e_k}$ appears 
on the boundary.  There are 
two edges on the left and right between the ends of $W_k$ 
and the first and last new copies of $c_j$.  These will 
be glued to other parts of the cyclic fatgraph.  
Finally, there are $o_j-e_k-1$ edges on the bottom in between 
the new copies of $c_j$.  Onto all of these edges, we attach a copy of 
the standard $\SS$ polygon.  These polygons have many edges 
remaining unglued, and we will return to them later.  
See Figure~\ref{fig:W_k_partial_fatgraphs}.

\begin{figure}[ht]
\labellist
\small\hair 2pt
 \pinlabel {$c_0$} at 195 54
 \pinlabel {$c_1$} at 153 59
 \pinlabel {$c_1$} at 123 59
 \pinlabel {$c_2$} at 81 54
 \pinlabel {$c_1$} at 24 54
 \pinlabel {$c_1$} at 8 26
 \pinlabel {$c_1$} at 39 26
 \pinlabel {$c_2$} at 55 28
 \pinlabel {$c_2$} at 80 10
 \pinlabel {$c_2$} at 104 27
 \pinlabel {$c_1$} at 137 28
 \pinlabel {$c_0$} at 176 26
 \pinlabel {$c_0$} at 210 23
 
 \pinlabel {$c_2$} at 15 5
 \pinlabel {$c_0$} at 33 4
 \pinlabel {$c_0$} at 55 16
 \pinlabel {$c_1$} at 64 6
 \pinlabel {$c_0$} at 96 6
 \pinlabel {$c_1$} at 105 16
 \pinlabel {$c_1$} at 186 5
 \pinlabel {$c_2$} at 203 5
\endlabellist
\includegraphics[scale=1.4]{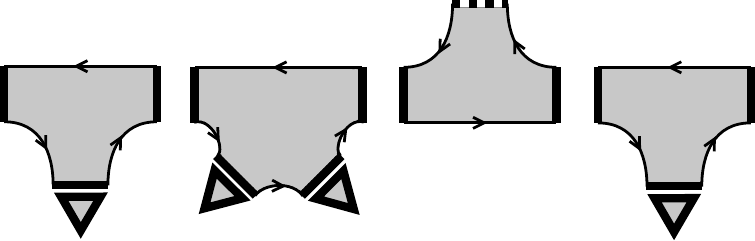}
\caption{Building a partial cyclic fatgraph for every run $W_k$ in 
$w = c_0c_1^2c_2c_1$.  
For this example, we let the generators of $\Gamma$ be 
$c_0, c_1, c_2$ of orders $3$, $3$, and $4$, respectively, and 
we take the cyclic order to be $O_\SS = [c_0, c_1, c_2]$.  
These partial fatgraphs are not glued together yet.
For clarity, we have labeled the unglued polygon edges with a single 
letter, so e.g. $pe(c_j)$ has label $c_j$.}
\label{fig:W_k_partial_fatgraphs}
\end{figure}
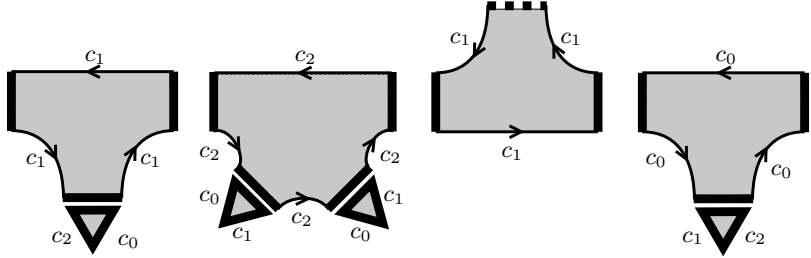

We have one partial fatgraph for every run $W_k$, 
and now we will insert a polygon between successive $W_k$ to 
glue them together.  The key is to do this in a way that 
only uses polygons allowed by Proposition~\ref{prop:immersion} 
and will also allow us to ensure that the boundary other than $w$ is 
copies of $b$.
To glue $W_k$ to $W_{k+1}$ (recall $W_k$ is to the 
right of $W_{k+1}$), suppose that $W_k$ is a run 
of generator $c_j$ and $W_{k+1}$ is a run of generator $c_l$.  
Let $P_k$ be the polygon whose boundary $\partial P_k$ 
is the interval in $O_\SS$ between $c_l$ and $c_j$.  
That is, $\partial P_k$ 
is a piece of the cyclically ordered set $S$, and $P_k$ 
is simply the standard $\SS$ polygon with some edges removed 
so that $c_j$ is followed by $c_l$.  If $c_l$ is immediately 
followed by $c_j$ in $O_\SS$, then $P_k$ will be a bigon.  

Now glue every $W_k$ to $W_{k+1}$ using $P_k$ in the middle. 
The result is a partial cyclic fatgraph, and observe that 
along the top, we have the word $w$.
On the far left and right ends, there remain two 
unglued group polygon edges.  On the left, build the 
polygon $P_K$ whose boundary is the interval in 
$O_\SS$ between $c_0$ and $w_m$; that is, 
the first edge of $P_K$ is $pe(c_0)$, and the 
last edge is $pe(w_m)$.  
For example, if $c_0 = w_m$, then $P_K$ will be degenerate (a monogon).
If $c_0$ follows $w_m$ in $O_\SS$, then $P_K$ will be a bigon, and so on.
Similarly, build the polygon $P_0$ whose boundary is the 
interval in $O_\SS$ between $w_0$ and $c_{J-1}$.  
Glue $P_K$ and $P_0$ on the left and right, respectively.  
Call the resulting partial fatgraph $Y'$.
See Figure~\ref{fig:W_k_glued_up}.

\begin{figure}[ht]                                                                                                                                       
\labellist                                                                                                                                                
\small\hair 2pt                                                                                                                                           
 \pinlabel {$c_0$} at 211 54
 \pinlabel {$c_1$} at 165 61                                                                                                                            
 \pinlabel {$c_1$} at 128 61                                                                                                                            
 \pinlabel {$c_2$} at 76 55
 \pinlabel {$c_1$} at 26 54
 
 \pinlabel {$c_0$} at -4 42                                                                                                                              
 \pinlabel {$c_1$} at 9 26                                                                                                                              
 \pinlabel {$c_2$} at 19 6                                                                                                                              
 \pinlabel {$c_0$} at 34 5                                                                                                                              
 \pinlabel {$c_1$} at 42 26                                                                                                                             
 \pinlabel {$c_2$} at 53 28                                                                                                                             
 \pinlabel {$c_0$} at 50 14                                                                                                                             
 \pinlabel {$c_1$} at 60 5.5                                                                                                                              
 \pinlabel {$c_2$} at 76 10.5                                                                                                                             
 \pinlabel {$c_0$} at 90 6.5                                                                                                                              
 \pinlabel {$c_1$} at 101 16                                                                                                                            
 \pinlabel {$c_2$} at 101 30                                                                                                                             
 \pinlabel {$c_0$} at 114 30                                                                                                                            
 \pinlabel {$c_1$} at 145 31                                                                                                                            
 \pinlabel {$c_2$} at 178 30                                                                                                                            
 \pinlabel {$c_0$} at 194 25                                                                                                                            
 \pinlabel {$c_1$} at 202 5
 \pinlabel {$c_2$} at 218 4
 \pinlabel {$c_0$} at 226 26
 \pinlabel {$c_1$} at 243 34
 \pinlabel {$c_2$} at 243 49
\endlabellist
\includegraphics[scale=1.35]{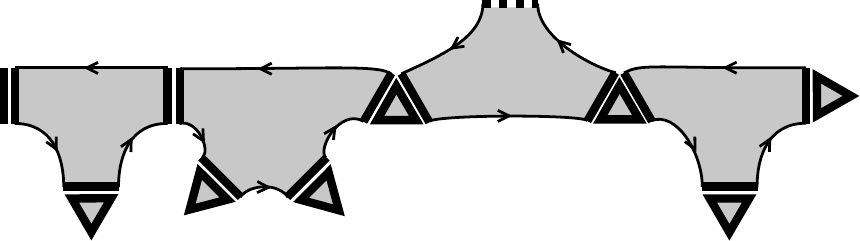}
\caption{The partial cyclic fatgraph $Y'$.  Again 
$pe(c_j)$ is labeled by $c_j$.}
\label{fig:W_k_glued_up}
\end{figure}
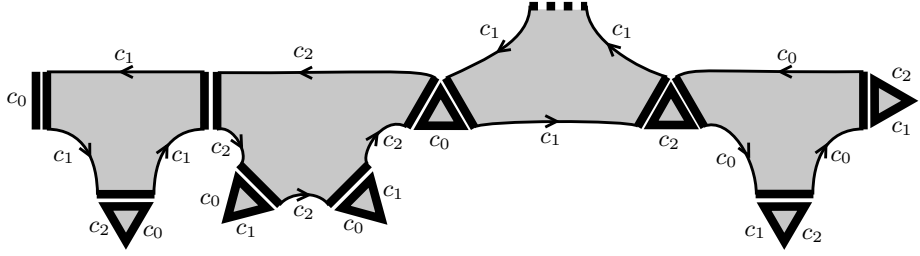

{\bf Step 2 (the covering trick):} At this point, we have the fatgraph $Y'$, 
which is composed of group polygons attached 
around polygons, with some of the polygon edges unglued.  
Also, all of the polygons have boundary which is an interval 
in the cyclic order $O_\SS$, so in particular all the polygons are 
small and have boundary cyclic order compatible with $O_\SS$.
In this step, we show how taking multiple copies of $Y'$ 
makes it possible to glue up all the unglued edges.

Consider what the boundary of $Y'$ is.  It is 
an oriented simplicial $1$-complex whose simplices 
are either labeled sides of group polygons  
or unglued edges from polygons.  
We claim the following (see Figure~\ref{fig:W_k_glued_up}):

\begin{lemma}\label{lem:boundary_with_unglued}
If we read each polygon edge $pe(c_j)$ 
as the generator $c_j$, then $\partial Y' = wb^m$ for some $m$.
\end{lemma}
\begin{proof}
This lemma is really just by construction.
Recall that the boundary of the standard $\SS$ polygon 
is the same as the boundary word of $\SS$, and every polygon in 
$Y'$ is an interval in $O_\SS$.  
Therefore, for any simplex in $\partial Y'$ which is
not part of $w$ and is labeled $pe(c_j)$ or $c_j$, 
the next simplex must be $pe(c_{j+1})$ or $c_{j+1}$, 
with indices modulo $J$.  If the simplex in $\partial Y'$ 
is part of $w$, then by construction it is a labeled 
side and is followed by the correct next letter.  The 
special cases of the last letter of $w$ and the last 
letter of $b$ before $w$ are also correct by construction.
Therefore, $\partial Y' = wb^m$, as desired.
\end{proof}
Lemma~\ref{lem:boundary_with_unglued} shows that \emph{if} the each 
unglued polygon edge $pe(c_j)$ is read as $c_j$, then the boundary is $wb^m$ 
for some $m$.  But of course this isn't enough --- we need to 
produce a complete fatgraph with real boundary.  
The trick is to take multiple copies of $Y'$ 
and attach group polygons in such a way that the unglued polygon edges 
are effectively replaced by labeled group polygon sides.  We now 
explain this trick.

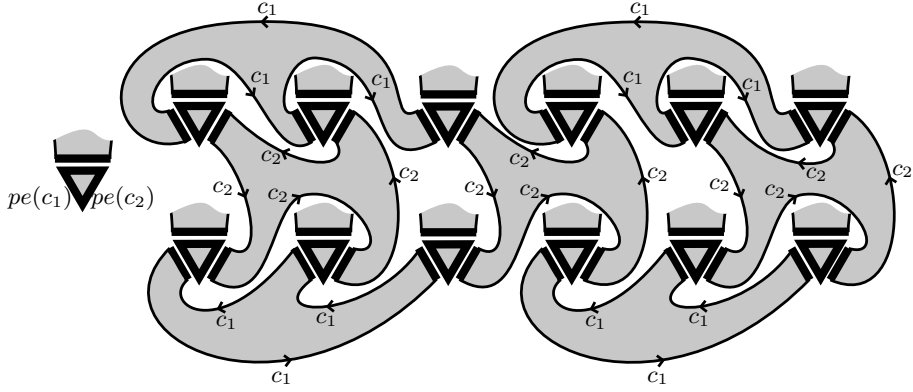
\begin{figure}[htb]
\begin{center}
\labellist
\small\hair 2pt
 \pinlabel {$pe(c_1)$} at 6 49
 \pinlabel {$pe(c_2)$} at 30 49
 
 \pinlabel {$c_1$} at 60 13
 \pinlabel {$c_1$} at 89 14
 \pinlabel {$c_1$} at 76 -3
 
 \pinlabel {$c_1$} at 167 13
 \pinlabel {$c_1$} at 197 14
 \pinlabel {$c_1$} at 185 -3
 
 \pinlabel {$c_1$} at 212.5 82
 \pinlabel {$c_1$} at 180 104
 \pinlabel {$c_1$} at 178 84
 
 \pinlabel {$c_1$} at 105 83
 \pinlabel {$c_1$} at 72 104
 \pinlabel {$c_1$} at 70 84
 
 \pinlabel {$c_2$} at 59 52
 \pinlabel {$c_2$} at 75 50
 \pinlabel {$c_2$} at 73 61
 \pinlabel {$c_2$} at 113 56
 
 \pinlabel {$c_2$} at 132 52
 \pinlabel {$c_2$} at 148 51
 \pinlabel {$c_2$} at 145 60
 \pinlabel {$c_2$} at 185 56
 
 \pinlabel {$c_2$} at 203 53
 \pinlabel {$c_2$} at 219 51
 \pinlabel {$c_2$} at 231 56
 \pinlabel {$c_2$} at 256 57
\endlabellist
\includegraphics[scale=1.3]{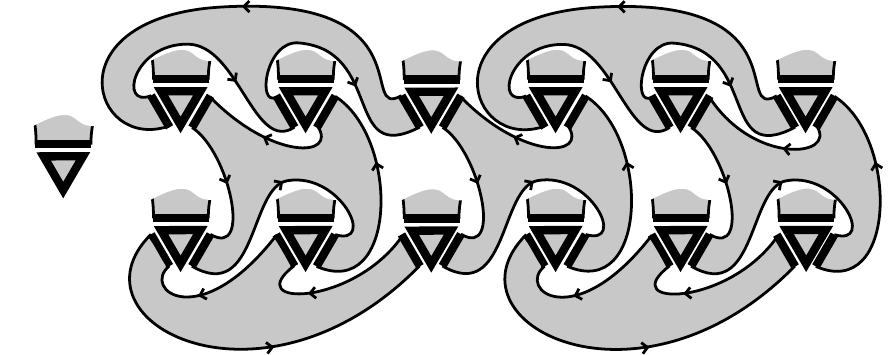}
\caption{Performing the covering trick on part of the partial 
fatgraph $Y'$ in Figure~\ref{fig:W_k_glued_up}.  
We take $12$ copies of $Y'$ and glue group polygons 
within the fibers.  This figure shows what happens to the group polygon on 
the lower right of Figure~\ref{fig:W_k_glued_up}.  The polygon (shown on left) 
is covered by $12$ copies.  If we glue in group polygons, note 
that the boundary has changed from many copies of $\ldots pe(c_1) pe(c_2) \ldots$ 
to many copies of $\ldots c_1 c_2 \ldots$.}
\label{fig:covering_trick}
\end{center}
\end{figure}

Let $L$ be the least common multiple of the $o_j$, and take 
$L$ copies of $Y'$.  Let $Y''$ be these $L$ copies of $Y'$, 
and think of $Y''$ 
as an $L$-sheeted cover of $Y'$.  The fiber over a single 
unglued edge $pe(c_j)$ in $Y'$ is $L$ copies of $pe(c_j)$. 
Attach $L/o_j$ group polygons to the $L$ 
unglued edges in $Y''$ in the fiber over a single unglued edge in $Y'$.
Each group polygon can be glued to arbitrary edges within the fiber.
Do this for every unglued edge in $Y'$.  We have attached 
many group polygons to the partial fatgraph $Y''$.  Call the 
result of attaching these group polygons $Y$.  We 
claim:

\begin{lemma}\label{lem:covering_trick}
The result $Y$ is a complete fatgraph whose boundary maps to $wb^m$, 
covering it $L$ times, and 
contains only small polygons whose boundary is compatible with $O_\SS$.
\end{lemma}
\begin{proof}
Every polygon in $Y$ is a polygon in $Y'$, which contained only small 
polygons with boundary compatible with $O_\SS$, 
so the last two conclusions are immediate.  
Also, $Y$ is created by gluing group polygons to all the unglued 
edges in $Y''$, so it is a complete fatgraph.  So the 
only question is what the boundary of $Y$ is.  
By Lemma~\ref{lem:boundary_with_unglued}, the boundary 
of $Y'$ is $wb^m$ when we read an unglued edge $pe(c_j)$ 
as the generator $c_j$, so since $Y''$ is an $L$-sheeted cover, 
the same thing is true for each sheet.  Now consider $Y$.  It is 
obtained from $Y''$ by gluing in group polygons within fibers 
over each unglued polygon edge in $Y''$.  So by construction, 
the boundary of $Y$ is obtained from $\partial Y''$ 
by taking the boundary loops in $Y''$ and replacing 
each unglued polygon edge $pe(c_j)$ with a group polygon side 
labeled by $c_j$ 
which runs between sheets of $Y''$.  
Thus $\partial Y$ consists of an $L$-degree cover of the loop $wb^m$.
\end{proof}

\begin{remark}
The $L$-sheeted cover $Y''$ of $Y'$ 
has $L$ separate boundary components.  When we glue on 
group polygons to obtain $Y$, each occurrence of 
an unglued edge $pe(c_j)$ 
is replaced by the labeled side $c_j$, as we desire.  However, 
this labeled side transits between two different sheets of $Y''$.
Therefore, the final boundary of $Y$ will look \emph{locally}
as if we simply replaced $pe(c_j)$ with $c_j$, but 
some of the boundary loops may have been joined together.
The total degree remains $L$.
\end{remark} 

For example, to perform this ``covering trick'' on the partial 
fatgraph $Y'$ in Figure~\ref{fig:W_k_glued_up}, 
we would compute the least common multiple of the 
$o_j$ $3$, $3$, and $4$, which is $12$.  Take 
$12$ copies of the partial fatgraph to get $Y''$.  
Now for any unglued polygon edge in $Y'$, glue group 
polygons to the unglued edges in $Y''$ in the fiber.
See Figure~\ref{fig:covering_trick}.

Because $Y$ contains only polygons 
which are compatible with the cyclic order $O_\SS$, we can apply 
Proposition~\ref{prop:immersion} to show that $f_Y$ is homotopic 
to an immersion with geodesic boundary, so $wb^m$ virtually 
bounds an immersed surface.  This completes Step $2$

{\bf Step 3:}
 
We need more than the fact that $wb^m$ virtually 
bounds an immersed surface: Theorem~\ref{thm:disk_orbifolds} is a stability 
result, and we need to show that there is $N$ such that 
for all $n>0$, we have that $wb^{N+ng}$ virtually bounds an immersed surface, 
where $g = \gcd(o_0-1,\ldots, o_{J-1}-1)$.
In Step $2$, we showed that for \emph{some} $m$, $wb^m$ virtually 
bounds an immersed surface.  
In this step, we show that we can actually achieve any desired 
$m$, as long as it is large enough and $g \mid m$.

Consider again the partial fatgraph $Y'$ from Step $2$.  By 
Lemma~\ref{lem:boundary_with_unglued}, if we read the 
unglued polygon edge $pe(c_j)$ as $c_j$, then the boundary 
of $Y'$ is $wb^m$ for some $m$.  For the current step, we need 
there to be some unglued polygon edge.  This is almost certainly 
the case, but if not, append a copy of $b$ onto $w$, which forces 
some unglued edges.  So without loss of generality, we assume there is 
an unglued polygon edge, and also without loss of generality, 
we assume it is $pe(c_0)$.  We will now re-use notation and 
define a new $Y''$ for this step.  Let $Y''$ be the partial fatgraph 
obtained from $Y'$ by attaching a $c_0$-group polygon onto the 
unglued polygon edge $pe(c_0)$ and attaching standard $\SS$ 
polygons onto all unglued edges of this new group polygon.

Note that we have added $o_0-1$ new polygons, and, reading 
unglued edges $pe(c_j)$ as $c_j$, the boundary of $Y''$ 
is $wb^{m+(o_0-1)}$.  Also note that the newly attached standard $\SS$ 
polygons have every edge unglued \emph{except} 
$c_0$.  So for any $j$ \emph{except} $j=0$, we 
can repeat this procedure to obtain a fatgraph whose boundary is 
$wb^{m+(o_0-1)+(o_j-1)}$ (when unglued edges are read as generators).
See Figure~\ref{fig:adding_group_polygons}.

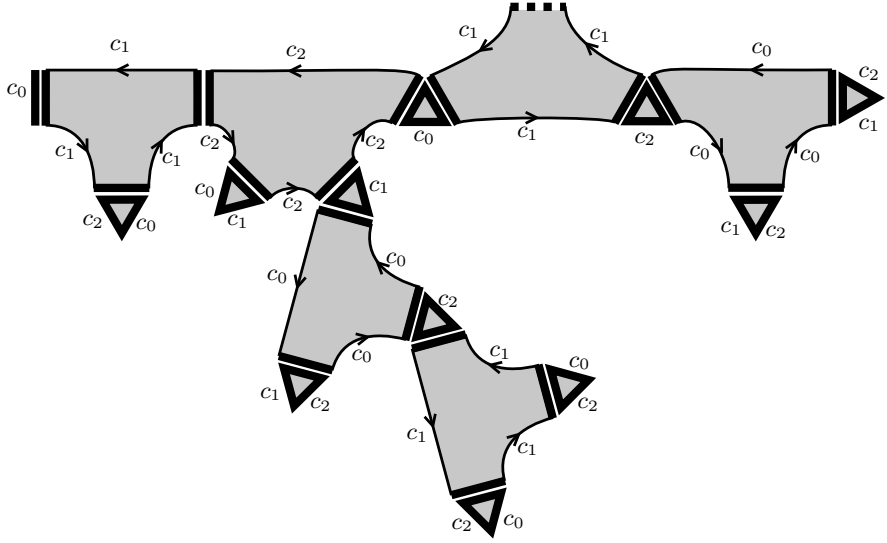
\begin{figure}[ht]
\labellist
\small\hair 2pt
 \pinlabel {$c_0$} at 212 142
 \pinlabel {$c_1$} at 165 146
 \pinlabel {$c_1$} at 127 147
 \pinlabel {$c_2$} at 77 141
 \pinlabel {$c_1$} at 26 143
 
 \pinlabel {$c_0$} at -4 130
 \pinlabel {$c_1$} at 10 112
 \pinlabel {$c_2$} at 18 92
 \pinlabel {$c_0$} at 33.5 91
 \pinlabel {$c_1$} at 41 110
 \pinlabel {$c_2$} at 52 115
 \pinlabel {$c_0$} at 50.5 101
 \pinlabel {$c_1$} at 60 92
 \pinlabel {$c_2$} at 76 97
 \pinlabel {$c_0$} at 71.5 76
 \pinlabel {$c_1$} at 69.5 42
 \pinlabel {$c_2$} at 84 39
 \pinlabel {$c_0$} at 96 53
 \pinlabel {$c_1$} at 112 30
 \pinlabel {$c_2$} at 125 4
 \pinlabel {$c_0$} at 140 5
 \pinlabel {$c_1$} at 144 26
 \pinlabel {$c_2$} at 161.5 39
 \pinlabel {$c_0$} at 159 51
 \pinlabel {$c_1$} at 136 54
 \pinlabel {$c_2$} at 121 69
 \pinlabel {$c_0$} at 106 81
 \pinlabel {$c_1$} at 101 102
 \pinlabel {$c_2$} at 100 115
 \pinlabel {$c_0$} at 114 116.5
 \pinlabel {$c_1$} at 144 117.5
 \pinlabel {$c_2$} at 178 116
 \pinlabel {$c_0$} at 194 112
 \pinlabel {$c_1$} at 203 90
 \pinlabel {$c_2$} at 216 89
 \pinlabel {$c_0$} at 226 111
 \pinlabel {$c_1$} at 243 120
 \pinlabel {$c_2$} at 243 135
\endlabellist
\includegraphics[scale=1.3]{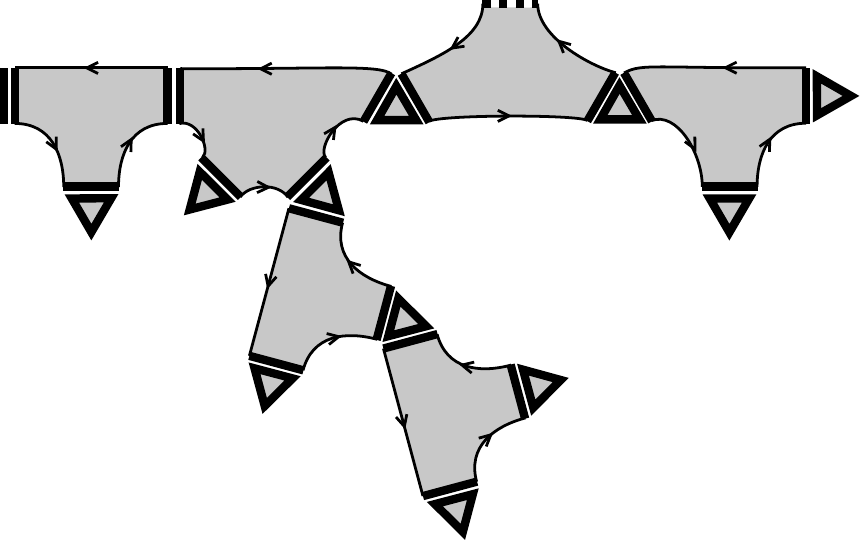}
\caption{Adding group polygons to add copies of $b$ to the boundary.
Notice that upon adding the first group polygon, the newly introduced 
standard $\SS$ polygons have every possible unglued edge \emph{except} $e(c_0, c_0)$, to which 
we attached the first group polygon.}
\label{fig:adding_group_polygons}
\end{figure}

Therefore, by repeating this procedure, 
we can obtain a partial fatgraph whose boundary (with $pe(c_j)$ read as $c_j$) 
is $wb^M$, where $M$ is any integer of the form
$m + \sum_{k=1}^K (o_{i_k}-1)$, where the successive $i_k$ are \emph{distinct}.
By Lemma~\ref{lem:number_theory}, there is some $N$ such that for all $n\ge 0$, 
every integer $N+ng$ is of this form.  Now take this partial 
fatgraph with boundary $wb^M$ and perform Step 2 (the covering trick) 
to get a real, complete fatgraph satisfying Proposition~\ref{prop:immersion}.
This shows that $wb^M$ virtually bounds an immersed surface, and completes 
the proof.
\end{proof}

The following lemma is required by the proof of Theorem~\ref{thm:disk_orbifolds}, 
but it is independently interesting.
\begin{lemma}\label{lem:number_theory}
Given integers $\{x_i\}_{i=1}^k$ with $k\ge 3$ and with $g = \gcd_ix_i$, there is 
some $N \in \N$ so that for all $n \in \N$, there is an integer 
sequence $\{i_j\}_{j=1}^J$ so that $i_j \ne i_{j+1}$ for all $j$ and 
so that $\sum_{j=1}^J x_{i_j} = N + ng$.
\end{lemma}
\begin{proof}
We call the sequence $1, 2, \ldots, k$ a \emph{run}.
The idea in constructing the sequence $i_j$ is to start with a
repeating sequence of runs $1, 2, \ldots, k, 1, 2, \ldots, k, \ldots$ 
and take away a single number from every other run.  The resulting 
list will have distinct adjacent pairs.

We can write $g = \sum_{i=1}^k a_ix_i$.  
Let $s = \sum_{i=1}^kx_i$, and let $M = \max_ia_i$ and $m = \min_ia_i$.
Let $N = 2ks(s/g)(M-m)$.  Now, given $C > N$ such that $g \mid C$, 
write $C = cs + r$, where $r < s$.  Notice that
\begin{itemize}
\item $c \ge 2k(s/g)(M-m)$, and 
\item $g \mid r$ because $g \mid s$, so
\item $r = dg$, where $d < s/g$.
\end{itemize}
Therefore, we can rewrite 
\begin{align*}
C &= cs + dg \\
  &= cs + d\sum_{i=1}^ka_ix_i \\
  &= cs + dM\sum_{i=1}^kx_i - d\sum_{i=1}^k(M-a_i)x_i \\
  &= (c + dM)s - \sum_{i=1}^kd(M-a_i)x_i, 
\end{align*}
where every term $M-a_i$ is non-negative.  Start with a sequence 
of $(c + dM)$ runs.  If we could remove $d(M-a_i)$ copies of $i$ 
from the sequence, for each $i$, then $\sum_jx_{i_j}$ for the resulting 
sequence $\{i_j\}_j$ would have the correct value $C$.  But we 
must be certain that when these copies of $i$ are removed, the adjacent
elements in the sequence remain distinct.  To accomplish this, we 
can remove a single $i$ from every other run.  We need to know that 
we have enough runs available, i.e. we need at least 
$2d(M-a_i) \le 2d(M-m) \le 2(s/g)(M-m)$ runs for every $i$.  
But from our bound on $c$, we have at least $c \ge 2k(s/g)(M-m)$ runs, so 
we can remove the indices as desired, and the proof is complete.
\end{proof}

\begin{remark}
\label{rem:generalization_of_CL}
Theorem~\ref{thm:disk_orbifolds} generalizes \cite{CL}, Theorem~3.1.
The situation of interest in \cite{CL} is $(2,p,\infty)$ orbifolds, 
which have only two orbifold points, and as stated, 
Theorem~\ref{thm:disk_orbifolds} requires $3$ orbifold points.  
However, in the special case of two orbifold points with one point 
of order $2$, Lemma~\ref{lem:number_theory} can be avoided, and 
Theorem~\ref{thm:disk_orbifolds} still goes through.  
The proof of Theorem~\ref{thm:disk_orbifolds} is essentially 
a combinatorialization of the argument in~\cite{CL}.
\end{remark}

\subsection{Orbifolds with genus}

We now prove an analog of Theorem~\ref{thm:disk_orbifolds} 
in the case that the orbifold has genus at least $1$.  
In this case, we can avoid any number-theoretic issues.

\begin{theorem}\label{thm:genus_orbifolds}
Let $\SS$ be a hyperbolic orbifold with one boundary component and 
with genus at least $1$.  Let $\Gamma = \pi_1(\SS)$ with $b = \partial \SS \in \Gamma$,
and let $w \in \Gamma$ be hyperbolic so that some power of $w$ is homologically trivial.  
Then there exists $N\in \N$ so that 
for all $n\ge 0$, the loop representing $wb^{N+n}$ virtually bounds 
an immersed surface with geodesic boundary.
\end{theorem}
\begin{proof}
As with the disk orbifold proof, in Step $1$, we construct
a partial fatgraph with some unglued polygon edges.  
The proof becomes different in Step $2$:
we cannot use a covering trick to fill in unglued 
polygon edges for infinite-order generators.  Therefore, 
we exhibit small partial fatgraph modules which can be 
inserted to fill in these edges.  Then we use the 
covering trick to fill in all the finite-order edges.

After relabeling, we can assume that the cyclic order on 
the infinite order generators is such that the boundary 
has the standard form $[z_0, z_1]\cdots [z_{I-2}, z_{I-1}]$ 
(if there are finite-order generators, they are 
inserted within this cyclic boundary word).
Therefore, the cyclic order is 
$[z_0,z_1^{-1},z_0^{-1},z_1, \cdots, z_{I-2},z_{I-1}^{-1},z_{I-2}^{-1},z_{I-1}]$ 
(with finite-order generators inserted at appropriate positions).  Note that 
the cyclic order isn't the same as the boundary word.
It will be useful to be able to refer to the 
letters in $w$ and the boundary word $b$, but their lengths 
will not matter.  Therefore, we use $w_0$ and $w_{-1}$ 
to refer to the first and last letters in $w$, and similarly for $b$.

{\bf Step 1:}
Perform Step~1 as in the proof 
of Theorem~\ref{thm:disk_orbifolds} to get a partial 
cyclic fatgraph $Y'$ which has boundary $w$ 
along the top and many unglued polygon edges.  
There are two situations not covered by those instructions, as follows.
First, each infinite order generator $z_i$ is in a run $W_k$ 
by itself, and the fatgraph piece we use for this run is simply a rectangle.
Second, when building the polygon to be glued onto the 
far left, we use the polygon which is 
the interval in $O_\SS$ between $b_0$ and $w_{-1}$, and 
the polygon for the far right is the interval between $w_0$ and $b_{-1}$.
In the proof of Theorem~\ref{thm:disk_orbifolds}, we used
$c_0$ and $c_{J-1}$ in place of $b_0$ and $b_{-1}$ because 
in the disk orbifold case we know that the boundary 
word is exactly $c_0\cdots c_{J-1}$.
See Figure~\ref{fig:W_k_genus_glued_up} for an example of $Y'$. 

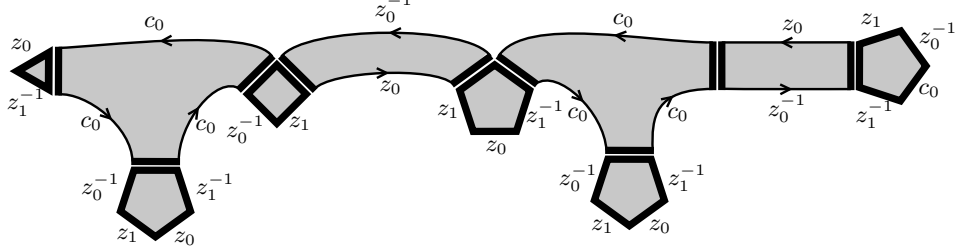
\begin{figure}[ht]
\labellist
\small\hair 2pt
 \pinlabel {$z_0$} at 259 71
 \pinlabel {$c_0$} at 203 73
 \pinlabel {$z_0^{-1}$} at 127 77
 \pinlabel {$c_0$} at 48 72
 
 \pinlabel {$z_0$} at 4 65
 \pinlabel {$z_1^{-1}$} at 5.5 46
 
 \pinlabel {$c_0$} at 27 39
 
 \pinlabel {$z_0^{-1}$} at 30 18
 \pinlabel {$z_1$} at 39 2
 \pinlabel {$z_0$} at 58 1
 \pinlabel {$z_1^{-1}$} at 68 18
 
 \pinlabel {$c_0$} at 65 38
 
 \pinlabel {$z_0^{-1}$} at 78 37
 \pinlabel {$z_1$} at 96 39
 
 \pinlabel {$z_0$} at 126 51
 
 \pinlabel {$z_1$} at 145 43
 \pinlabel {$z_0$} at 161 31
 \pinlabel {$z_1^{-1}$} at 177 42
 
 \pinlabel {$c_0$} at 188 42
 
 \pinlabel {$z_0^{-1}$} at 188 21
 \pinlabel {$z_1$} at 196 5
 \pinlabel {$z_0$} at 215 5
 \pinlabel {$z_1^{-1}$} at 224 21
 
 \pinlabel {$c_0$} at 221 42
 
 \pinlabel {$z_0^{-1}$} at 257 44
 
 \pinlabel {$z_1^{-1}$} at 286 42
 \pinlabel {$c_0$} at 304 50
 \pinlabel {$z_0^{-1}$} at 307 68
 \pinlabel {$z_1$} at 285 73
\endlabellist
\includegraphics[scale=1.15]{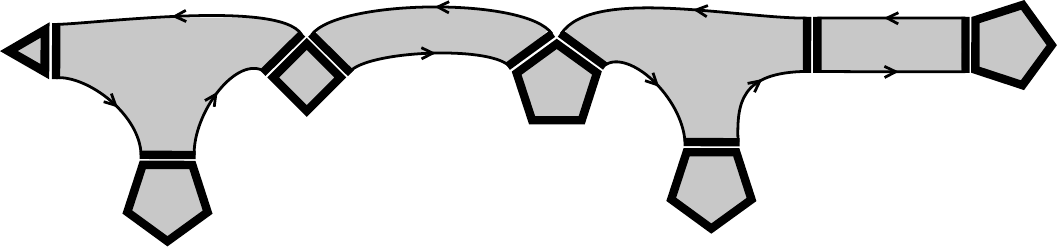}
\caption{The partial fatgraph $Y'$ for the word $w=z_0c_0z_0^{-1}c_0$ 
in an orbifold of genus $1$ with one orbifold point of order $3$, 
with cyclic order $[z_0, z_1^{-1}, c_0, z_0^{-1},z_1]$ and thus 
boundary $b=z_0z_1c_0z_0^{-1}z_1^{-1}$.  For 
simplicity, we denote the edges $pe(z_i)$ and $pe(c_j)$ by 
$z_i$ and $c_j$, respectively.  This picture omits the 
replacement of $w$ by $wb^2$, as described at the end of Step~$1$, 
because it is not necessary to do that to have enough unglued 
polygon edges in this example.}
\label{fig:W_k_genus_glued_up}
\end{figure}

For reasons which become apparent in Step~$3$, we need for there to
be sufficiently many unglued polygon edges.  Therefore, 
we assume that we have appended a copy of $b^2$ onto $w$.  
Since the polygon inserted between letters in $b$ is a 
complete standard $\SS$ polygon, and $b$ has length at least $4$, 
we ensure that there are at least four polygon edges $pe(z_i)$ and 
four polygon edges $pe(z_i^{-1})$ for each $i$.  
For simplicity, we will not show this in our example pictures.

{\bf Step 2:}

In this step, we describe how to fill in the 
unglued polygon edges associated with the 
infinite-order generators.  
We do this by building small partial fatgraph ``modules'' 
which can be glued in to complete $Y'$. 

It will be convenient to be able to refer to subwords of the boundary 
word $b$.  Denote by $b_{i,+}$ and $b_{i,-}$ the subword of the 
(cyclic) word $b$ between (not including) $z_i$ and $z_i^{-1}$, and 
$z_i^{-1}$ and $z_i$, 
respectively.  For example, if $b=z_0z_1c_0z_0^{-1}z_1^{-1}$, then 
$b_{0,+} = z_1c_0$ and $b_{0,-} = z_1^{-1}$.

Now let $P$ be the standard $\SS$ polygon, whose boundary is 
$O_\SS$.  For each infinite-order generator $z_i$, 
attach the rectangle $r(z_i)$ to both polygon edges 
$pe(z_i)$ and $pe(z_i^{1-})$.
Call this partial fatgraph $A$.  Note that if we read 
unglued finite-order polygon edges $pe(c_j)$ as $c_j$, 
then $\partial A = b$; i.e., the boundary of $A$ is 
the boundary word of $\SS$.  

For each $i$, define a partial fatgraph $A_i$ as follows.
Detach one edge of the rectangle $r(z_i)$ 
from the polygon $P$, and glue a duplicate copy of 
$r(z_i)$ to the polygon edge which is now unglued.
That is, instead of there being a single rectangle with both 
edges glued to $P$, there are now two rectangles, each of which is 
glued to one of the two edges $pe(z_i)$ and $pe(z_i^{-1})$ 
in $P$.   See Figure~\ref{fig:W_k_modules}.

\begin{figure}[htb]
\begin{center}
\labellist
\small\hair 2pt
 \pinlabel {$z_0$} at 88 49
 \pinlabel {$z_1$} at 32 69
 \pinlabel {$z_1^{-1}$} at 36 52
 \pinlabel {$z_0$} at 12 49
 \pinlabel {$re(z_0^{-1})$} at -13 37
 \pinlabel {$z_0^{-1}$} at 10 22
 \pinlabel {$c_0$} at 65 17
 \pinlabel {$z_0^{-1}$} at 88 22
 \pinlabel {$re(z_0)$} at 122 33
 
 \pinlabel {$z_1$} at 240 49
 \pinlabel {$c_0$} at 218 54
 \pinlabel {$z_0^{-1}$} at 163 62
 \pinlabel {$z_0$} at 189 48
 \pinlabel {$z_1$} at 164 48
 \pinlabel {$z_1^{-1}$} at 166 21
 \pinlabel {$z_1^{-1}$} at 240 22
 \pinlabel {$re(z_1)$} at 274 36
 \pinlabel {$re(z_1^{-1})$} at 139 41
\endlabellist
\includegraphics[scale=1.2]{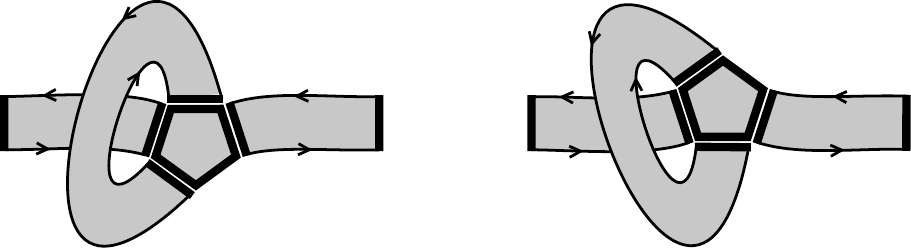}
\caption{The partial fatgraph modules $A_0$ and $A_1$ in the 
example from Figure~\ref{fig:W_k_genus_glued_up}.}
\label{fig:W_k_modules}
\end{center}
\end{figure}
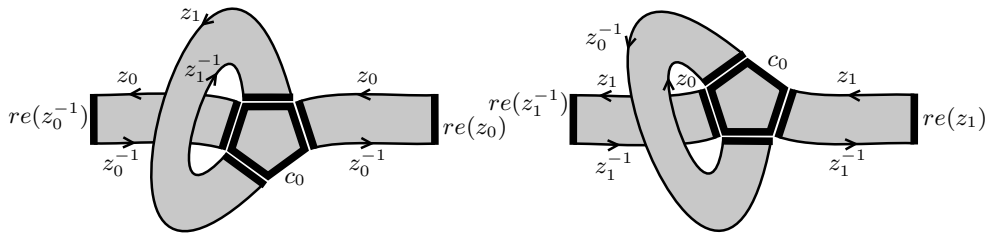

\begin{lemma}\label{lem:fatgraph_modules}
The partial fatgraph $A_i$ has two boundary components.  
When finite-order polygon edges $pe(c_j)$ are read as $c_j$, 
the $1$-simplicies in the boundary are labeled 
$pe(z_i)z_ib_{i,+}z_i^{-1}$ and $pe(z_i^{-1})z_i^{-1}b_{i,-}z_i$.
Consequently, there 
are no infinite-order unglued polygon edges, and there are exactly 
the two unglued rectangle edges $re(z_i)$ and $re(z_i^{-1})$.
\end{lemma}
\begin{proof}
This is just by construction; Figure~\ref{fig:W_k_modules} 
illustrates it.
Suppose we build a fatgraph $B$ by attaching the 
bigon $[pe(z_i),pe(z_i^{-1})]$ 
to the two unglued rectangle edges in $A_i$.  Then 
we have effectively taken $A$ and replaced the rectangle $r(z_i)$ 
with two rectangles glued end-to-end.  Therefore, the boundary 
of $B$ is $b$ with $z_i$ and $z_i^{-1}$ duplicated.  So 
the boundary of $B$ is $z_i^2b_{i,+}z_i^{-2}b_{i,-}$.  Now remove the bigon to get $A_i$ 
back from $B$, which cuts the boundary in two at the $z_i^2$ and $z_i^{-2}$ 
and inserts the rectangle 
edges as claimed.
\end{proof}

The partial fatgraphs $A_i$ are the small modules which we will 
glue onto $Y'$ to fill in the unglued polygon edges.  Each $A_i$ 
has the two rectangle edges $re(z_i)$ and $re(z_i^{-1})$.  
So we must verify that $Y'$ contains the same number of 
$pe(z_i)$ as $pe(z_i^{-1})$ for each $i$.

\begin{lemma}\label{lem:equidistributed_edges}
For every $i$, $Y'$ contains the same number of 
unglued polygon edges $pe(z_i)$ and 
$pe(z_i^{-1})$.
\end{lemma}
\begin{proof}
The purpose of the this lemma is to verify that we 
can attach the partial fatgraphs $A_i$ to $Y'$ 
to fill in all the unglued polygon edges.  So it is interesting 
that to \emph{prove} this lemma, we will attach different partial 
fatgraphs to $Y'$ and then make some observations about the result.
It is also possible to prove the lemma with some technical 
combinatorial counting, but this method is more intuitive.

Note that the infinite-order generators naturally come in pairs, 
one for each genus.  For each $i$, let $i'$ denote 
the index with which $i$ is paired.  So $i' = i+1$ if $i\equiv 0 \mod 2$, 
and $i' = i-1$ if $i\equiv 1 \mod 2$.

Consider the partial fatgraph $A_i$.  It has two 
unglued rectangle edges $re(z_i)$ and $re(z_i^{-1})$.
Because the fundamental group of $A_i$ embeds in $\Gamma$, we will 
refer to elements of $\pi_1(A_i)$ by their images in $\Gamma$.
Because we have assumed the standard form for the generators, 
a loop freely homotopic to $z_{i'}$ in $A_i$ is separating, and 
the two boundary components of $A_i$ are in different connected 
components of the complement of the loop $z_{i'}$. 
See Figure~\ref{fig:small_modules}.

\begin{figure}[htb]
\begin{center}
\labellist
\small\hair 2pt
 \pinlabel {$z_0$} at 88 49
 \pinlabel {$z_1$} at 32 69
 \tiny
 \pinlabel {$z_1^{-1}$} at 46 46
 \small
 \pinlabel {$z_0$} at 12 49
 \pinlabel {$re(z_0^{-1})$} at -13 37
 \pinlabel {$z_0^{-1}$} at 10 22
 \pinlabel {$c_0$} at 65 17
 \pinlabel {$z_0^{-1}$} at 88 22
 \pinlabel {$re(z_0)$} at 122 33
 
 \pinlabel {$z_1$} at 240 49
 \pinlabel {$c_0$} at 218 54
 \pinlabel {$z_0^{-1}$} at 163 62
 \pinlabel {$z_0$} at 189 48
 \pinlabel {$z_1$} at 164 48
 \pinlabel {$z_1^{-1}$} at 166 21
 \pinlabel {$z_1^{-1}$} at 240 22
 \pinlabel {$re(z_1)$} at 274 36
 \pinlabel {$re(z_1^{-1})$} at 139 41
 
 \Large 
 \pinlabel {$A_{0,-}$} at 9 62
 \pinlabel {$A_{0,+}$} at 100 62
 \pinlabel {$A_{1,-}$} at 140 62
 \pinlabel {$A_{1,+}$} at 250 62
\endlabellist
\includegraphics[scale=1.2]{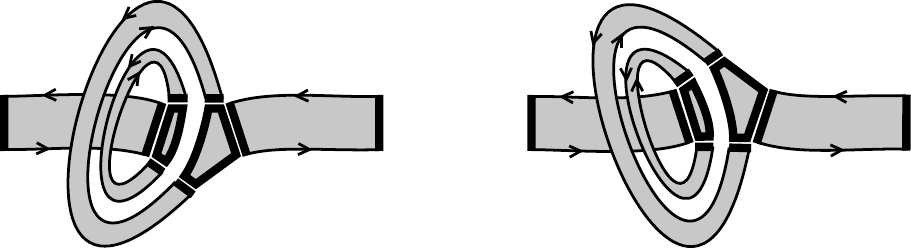}
\caption{The fatgraph modules $A_0$ and $A_1$ as in 
Figure~\ref{fig:W_k_modules}, get cut into the modules $A_{0,\pm}$ 
and $A_{1,\pm}$ in the proof of Lemma~\ref{lem:equidistributed_edges}.}
\label{fig:small_modules}
\end{center}
\end{figure}
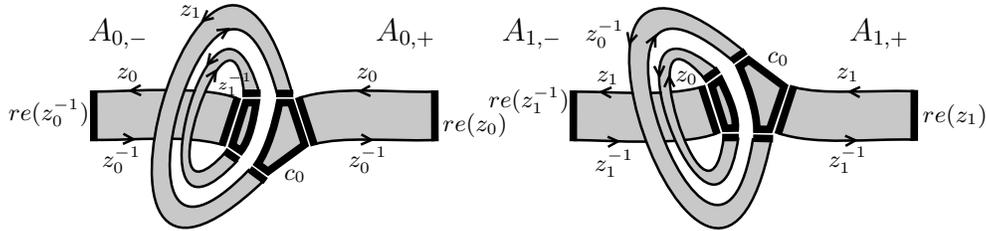

Thus the loop $z_{i'}$ cuts $A_i$ into two surfaces, which we will 
refer to as $A_{i,+}$ and $A_{i,-}$.  The surface $A_{i,+}$ has two boundary 
components: the loops $re(z_i)z_ib_{i,+}z_i^{-1}$ and $z_{i'}^{\pm 1}$, and the 
surface $A_{i,-}$ has the two boundaries $re(z_i^{-1})z_i^{-1}b_{i,-}z_i$ and 
$z_{i'}^{\mp 1}$.  We use $z_{i'}^{\pm 1}$ and $z_{i'}^{\mp 1}$ to 
refer to the positive and negative powers of $z_{i'}$ because 
which power goes with $A_{i,+}$ and $A_{i,-}$ 
depends on the parity of $i$.  This doesn't matter; the 
key fact is that $A_{i,+}$ and $A_{i,-}$ each have a $z_{i'}$ boundary 
loop, and the loops on $A_{i,+}$ and $A_{i,-}$ have opposite signs.

So for each $i$, we have the partial fatgraphs $A_{i,+}$ and $A_{i,-}$, 
which have the single unglued rectangle edge $re(z_i)$ and 
$re(z_i^{-1})$, respectively.
Build a new fatgraph $X$ by attaching a copy of $A_{i,\pm}$ to 
every unglued infinite-order polygon edge $pe(z_i^{\pm 1})$ in $Y'$.
Each $A_{i,\pm}$ has only a single unglued rectangle edge, so there is 
no obstruction to attaching them to every unglued edge.
Because the $A_{i,\pm}$ have no unglued infinite-order polygon edges, 
there are no infinite-order unglued polygon edges in $X$, so 
$X$ contains only labeled 
sides of rectangles and unglued finite-order polygon edges.
And by the construction of $Y'$ (because every polygon in $Y'$ is 
an interval in $O_\SS$), if we read each finite-order edge $pe(c_j)$ 
as $c_j$, the boundary of $X$ is of the form
$wb^m + \sum_{i=0}^{I-1} M_i z_i + m_i z_i^{-1}$ for integers $m$, $M_i$, and $m_i$.
The $z_i^{\pm 1}$ boundary components arise from the small boundary loops in 
the $A_{i,\pm}$.  Because the only unglued edges in $X$ are finite-order 
edges, the boundary of $X$ has a finite power which is homologically 
trivial.  And since $w$ and $b$ both have finite powers which are 
homologically trivial, the sum $\sum_{i=0}^{I-1} M_i z_i + m_i z_i^{-1}$ 
must be homologically trivial, so it must be that $m_i = M_i$ for each $i$.
But these integers count the number of copies of 
$A_{i,+}$ and $A_{i,-}$ which we attached to $Y'$, so we 
conclude that the number of unglued edges 
$pe(z_i)$ is equal to the number of unglued edges $pe(z_i^{-1})$, 
as desired.
\end{proof}

By Lemma~\ref{lem:equidistributed_edges}, 
the partial fatgraph $Y'$ contains the same number of edges $pe(z_i)$ 
as $pe(z_i^{-1})$ for each $i$.  Call this number $m_i$.  
Therefore, it is possible to attach $m_i$ copies of $A_i$ to 
$Y'$ for each $i$.  Call the resulting fatgraph $Y''$.  
By construction, $Y''$ has a single boundary component, 
which, if $pe(c_j)$ is read as $c_j$, is $wb^m$ for some $m$.
By applying Step~$2$ of Theorem~\ref{thm:disk_orbifolds}, the 
covering trick, to $Y''$ we can produce a complete fatgraph $Y$ whose boundary 
covers $wb^m$.  By construction $Y$ contains only polygons 
whose boundaries are intervals of $O_\SS$, so they are 
small and compatible with $O_\SS$.  Therefore, the 
existence of $Y$ shows that $wb^m$ virtually bounds an immersed 
surface in $\SS$.  
Figure~\ref{fig:W_k_genus_glued_up_with_As} shows the result $Y''$ 
of attaching all the $A_i$ to the partial fatgraph $Y'$
shown in Figure~\ref{fig:W_k_genus_glued_up}.
This completes Step~$2$.

\begin{figure}[htb]
\begin{center}
\labellist
\small\hair 2pt
 \pinlabel {$z_0$} at 316 200
 \pinlabel {$c_0$} at 261 202
 \pinlabel {$z_0^{-1}$} at 185 206
 \pinlabel {$c_0$} at 108 200
 
 \pinlabel {$z_0$} at 10 190
 \pinlabel {$z_1$} at 16 96
 \pinlabel {$c_0$} at 39 156
 \pinlabel {$z_0^{-1}$} at 40 178
 \tiny \pinlabel {$z_1^{-1}$} at 49 164 \small
 
 \tiny \pinlabel {$z_0$} at 135 111 \small
 \pinlabel {$z_1$} at 81 160
 \pinlabel {$c_0$} at 89 165
 \pinlabel {$z_0^{-1}$} at 52 93
 \tiny \pinlabel {$z_1^{-1}$} at 15 36 \small
 
 \tiny \pinlabel {$z_0$} at 62 67 \small
 \pinlabel {$z_1$} at 101 9
 \pinlabel {$c_0$} at 134 13
 \pinlabel {$z_0^{-1}$} at 183 40
 \tiny \pinlabel {$z_1^{-1}$} at 106 36 \small
 
 \pinlabel {$z_0$} at 66 34
 \pinlabel {$z_1$} at 5 48
 \pinlabel {$c_0$} at 41 9
 \pinlabel {$z_0^{-1}$} at 71 5
 \tiny \pinlabel {$z_1^{-1}$} at 108 66 \small
 
 \tiny \pinlabel {$z_0$} at 159 43 \small
 \pinlabel {$z_1$} at 107 93
 \pinlabel {$c_0$} at 123 166
 \tiny \pinlabel {$z_0^{-1}$} at 95 144 \small
 \tiny \pinlabel {$z_1^{-1}$} at 33 110 \small
 
 \tiny \pinlabel {$z_0$} at 113 139 \small
 \pinlabel {$z_1$} at 150 154
 \pinlabel {$c_0$} at 143 143
 \tiny \pinlabel {$z_0^{-1}$} at 135 89 \small
 \tiny \pinlabel {$z_1^{-1}$} at 177 169 \small
 
 \pinlabel {$z_0$} at 184 178
 \pinlabel {$z_1$} at 168 124
 \pinlabel {$c_0$} at 193 88
 \pinlabel {$z_0^{-1}$} at 253 85
 \tiny \pinlabel {$z_1^{-1}$} at 195 132 \small
 
 \pinlabel {$z_0$} at 258 104
 \pinlabel {$z_1$} at 340 102
 \pinlabel {$c_0$} at 283 133
 \tiny \pinlabel {$z_0^{-1}$} at 259 134 \small
 \tiny \pinlabel {$z_1^{-1}$} at 242 159 \small
 
 \tiny \pinlabel {$z_0$} at 224 94 \small
 \pinlabel {$z_1$} at 252 148
 \pinlabel {$c_0$} at 248 169
 \pinlabel {$z_0^{-1}$} at 271 86
 \tiny \pinlabel {$z_1^{-1}$} at 225 24 \small
 
 \pinlabel {$z_0$} at 272 57
 \pinlabel {$z_1$} at 314 0
 \pinlabel {$c_0$} at 345 4
 \pinlabel {$z_0^{-1}$} at 396 33
 \tiny \pinlabel {$z_1^{-1}$} at 318 27 \small
 
 \pinlabel {$z_0$} at 277 24
 \pinlabel {$z_1$} at 214 35
 \pinlabel {$c_0$} at 254 0
 \pinlabel {$z_0^{-1}$} at 284 -3
 \tiny \pinlabel {$z_1^{-1}$} at 320 57 \small
 
 \pinlabel {$z_0$} at 370 36
 \pinlabel {$z_1$} at 321 83
 \pinlabel {$c_0$} at 280 170
 \pinlabel {$z_0^{-1}$} at 315 172
 \pinlabel {$z_1^{-1}$} at 367 104
 
 \pinlabel {$z_0$} at 411 154
 \pinlabel {$z_1$} at 367 132
 \pinlabel {$c_0$} at 364 178
 \tiny \pinlabel {$z_0^{-1}$} at 359 169 \small
 \tiny \pinlabel {$z_1^{-1}$} at 320 118.5 \small
 
 \pinlabel {$z_0$} at 371 148
 \pinlabel {$z_1$} at 370 212
 \pinlabel {$c_0$} at 396 182
 \pinlabel {$z_0^{-1}$} at 423 134
 \pinlabel {$z_1^{-1}$} at 379 234
\endlabellist
\includegraphics[scale=0.8]{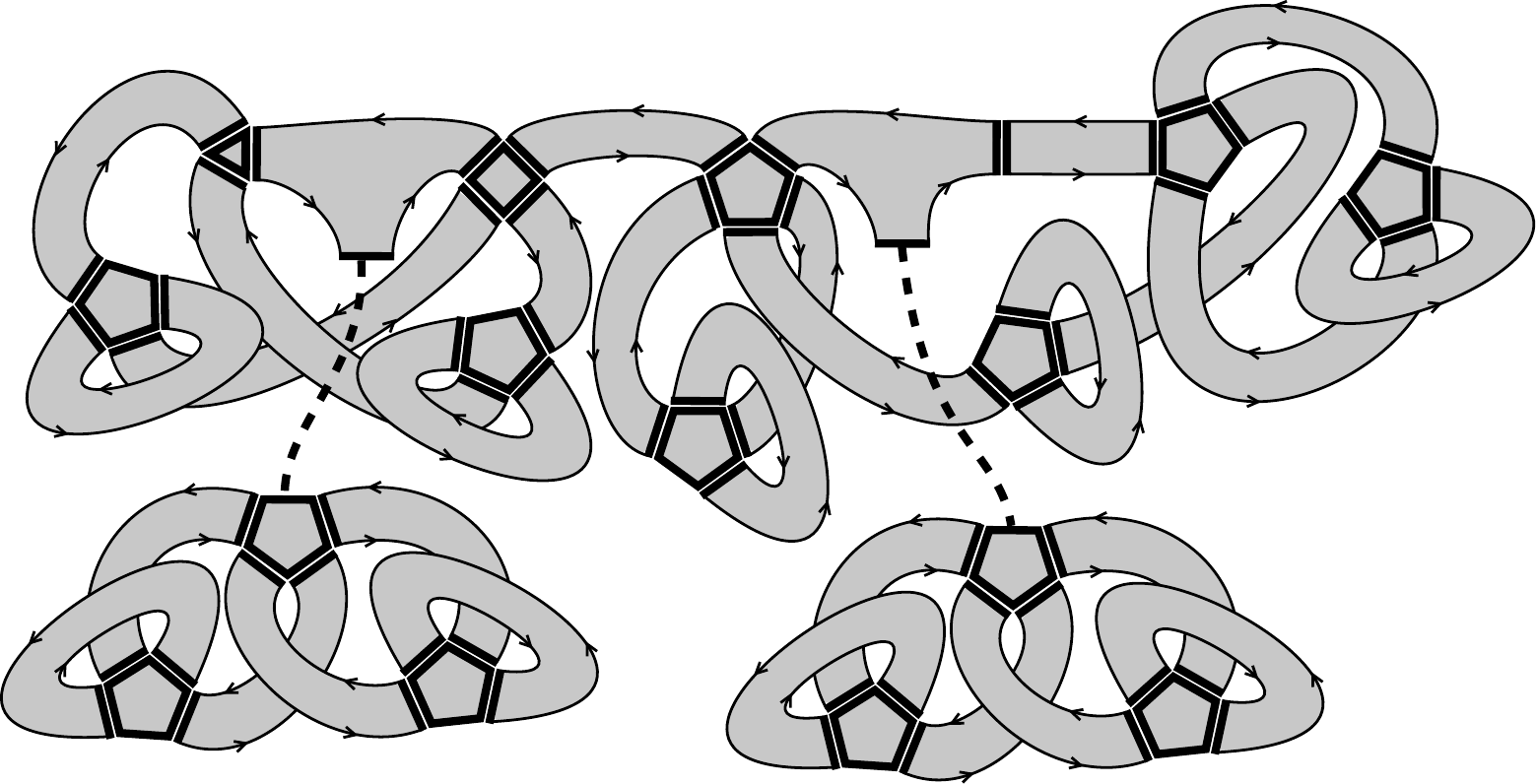}
\caption{The result of attaching the $A_i$ to 
the partial fatgraph in 
Figure~\ref{fig:W_k_genus_glued_up}.  For clarity, two 
regions of the fatgraph are drawn separately, and dotted 
lines indicate where they should be placed.  
Patience reveals that the 
boundary is, in fact, $wb^{14}$.}
\label{fig:W_k_genus_glued_up_with_As}
\end{center}
\end{figure}
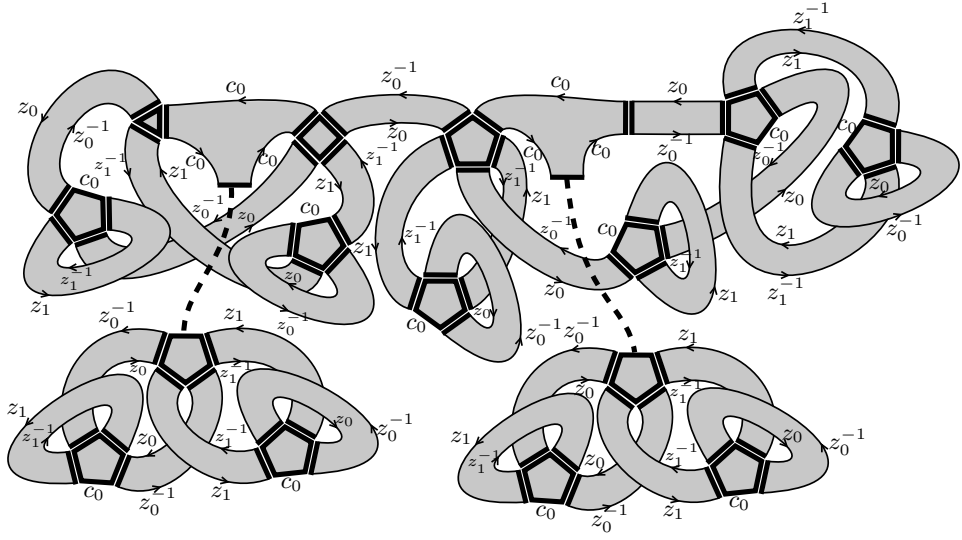

{\bf Step 3:}

As with the proof of Theorem~\ref{thm:disk_orbifolds}, 
we have shown that $wb^m$ virtually bounds an immersed 
surface for some $m$, but we need to show the stability result 
that there is $N$ such that for all $n\ge 0$, 
$wb^{N+n}$ virtually bounds an immersed surface.

To prove this, we will construct some new partial fatgraph ``modules'' 
which can replace some of the $A_i$ 
and have the effect of increasing the power of $b$ in the boundary.
As it turns out, it is quite simple to increase the power of $b$ 
by \emph{two}, which would prove the theorem for \emph{even} $n$.
It is more complicated to increase the power by one; this requires 
taking a cover.

First, we show how to increase the power of $b$ by 
(a multiple of) two.
This requires exhibiting a new kind of partial fatgraph, 
which we will denote by $A_{i,k}$.  
Recall that $i'$ is the generator paired with $i$, so 
that $[z_i,z_{i'}]^{\pm 1}$ appears in $b$, possibly with finite-order 
generators inserted.  
Given $k$ even, build $A_{i,k}$ as follows 
(See Figure~\ref{fig:increase_by_2k}): 
take $k+1$ copies of the standard $\SS$ 
polygon, indexed by $P_\ell$ for $\ell = 0 \ldots k$.  
For every infinite-order generator $z_t$ 
with $t \ne i$ and $t \ne i'$, add $k+1$ copies of the 
rectangle $r(z_t)$, each one connected at both edges to a single 
polygon $P_\ell$.  Next, add $k+2$ copies of $r(z_i)$.  
One copy has rectangle edge $re(z_i)$ connected to $P_0$; 
one copy has rectangle edge $re(z_i^{-1})$ connected to $P_k$; and 
the remaining copies connect $P_\ell$ to $P_{\ell+1}$.
Finally, add $k+1$ copies of $r(z_{i'})$, as follows:
for each $\ell$ divisible by $2$, add two copies of $r(z_{i'})$ 
which connect $P_\ell$ to $P_{\ell+1}$.  
By construction, the boundary of $A_{i,k}$ has two components, 
with labels $re(z_i^{-1})z_i^{-1}b^kb_{i,-}z_i$ and 
$re(z_i)z_ib_{i,+}z_i^{-1}$.  Note the power $b^k$ in the 
first boundary component is the \emph{cyclic} word $b^k$; it may be 
cyclically rotated from the original choice of a cyclic representative 
that we called $b$.  This is correct, since if we insert a copy of 
$b$ in the middle of a power of $b$, we must cyclically shift the 
inserted copy, depending on the location it is inserted, 
so that it aligns correctly.
See Figure~\ref{fig:increase_by_2k}.

\begin{figure}[htb]
\begin{center}\labellist
\small\hair 2pt
 \pinlabel {$z_0^{-1}$} at 12 42
 \pinlabel {$z_0^{-1}$} at 109 32
 \pinlabel {$z_0^{-1}$} at 147.5 43
 \pinlabel {$z_0^{-1}$} at 220 44
 
 \pinlabel {$z_0$} at 217 68.5
 \pinlabel {$z_0$} at 145 68
 \pinlabel {$z_0$} at 113.5 66
 \pinlabel {$z_0$} at 13 68.5
 
 \pinlabel {$re(z_0^{-1})$} at -14 57
 \pinlabel {$c_0$} at 65 38
 \pinlabel {$c_0$} at 136 39
 \pinlabel {$c_0$} at 198 39
 \pinlabel {$re(z_0)$} at 255 56
 
 \pinlabel {$z_1^{-1}$} at 71 8
 \pinlabel {$z_1$} at 59 28
 \pinlabel {$z_1$} at 74 74
 \pinlabel {$z_1^{-1}$} at 78 97
 \tiny
 \pinlabel {$z_1^{-1}$} at 171 74
 \small
 \pinlabel {$z_1$} at 166 92
\endlabellist
\includegraphics[scale=1.25]{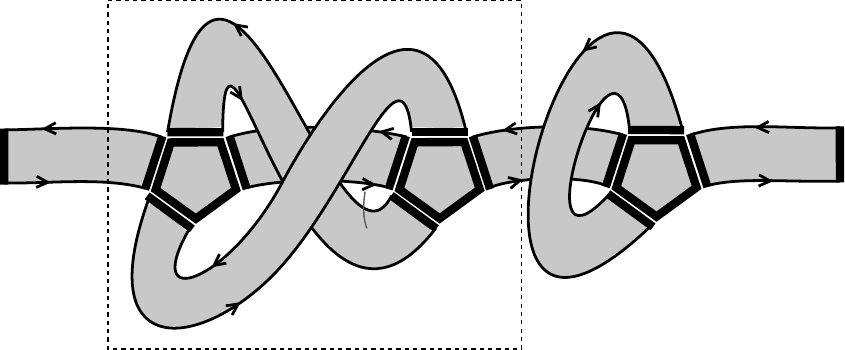}
\caption{The partial fatgraph $A_{0,2}$.  The boxed area 
is duplicated as desired to produce $A_{0,k}$.  
Note that replacing one of the $A_0$ with $A_{0,k}$ 
increases the power of $b$ in the boundary by $k$.}
\label{fig:increase_by_2k}
\end{center}
\end{figure}
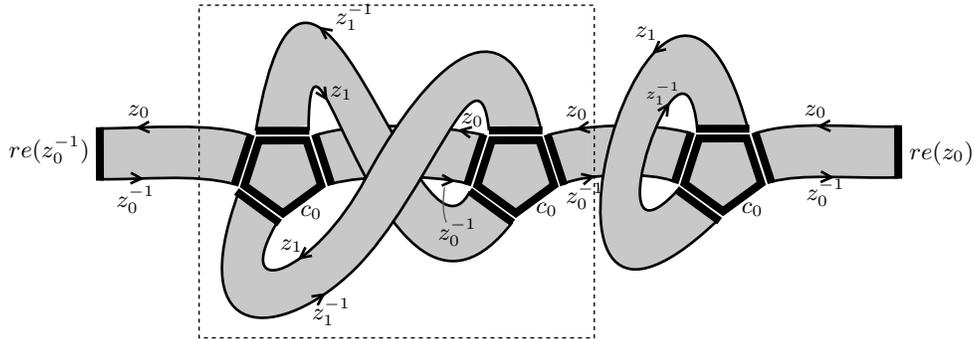

Note that the boundary of $A_{i,k}$ is exactly that of $A_i$, 
except one of the boundaries has $k$ copies of $b$ inserted.  
Thus, if we replace one of the copies of $A_i$ for some $i$ with 
$A_{i,k}$ in Step~$2$, the resulting fatgraph $Y$ 
has boundary $wb^{m+k}$.  The fact that 
we can insert a copy of $A_{i,k}$ requires that we 
have at least two unglued polygon edges.  Recall we 
ensured this in Step~$1$.  This shows that $wb^{m+k}$ 
virtually bounds an immersed surface for any even $k$.

\begin{figure}[htb]
\begin{center}
\labellist
\small\hair 2pt
 \pinlabel {$e_1$} at -6 48
 \small
 \pinlabel {$z_0$} at 20 26
 \pinlabel {$z_1$} at 95 -3
 \pinlabel {$c_0$} at 158 31
 \pinlabel {$z_0^{-1}$} at 176 38
 \tiny
 \pinlabel {$z_1^{-1}$} at 203 69
 \small
 \pinlabel {$z_0$} at 173 66
 \pinlabel {$z_1$} at 111 104
 \pinlabel {$c_0$} at 45 56
 \pinlabel {$z_0^{-1}$} at 23 54
 
 \pinlabel {$e_3$} at 114 67
 \small
 \pinlabel {$z_0^{-1}$} at 86 69
 \pinlabel {$z_1^{-1}$} at 112 81
 \pinlabel {$z_0$} at 120 57.5
 \pinlabel {$e_3'$} at 98 35
 \small
 \pinlabel {$z_0^{-1}$} at 121 29
 \pinlabel {$z_1^{-1}$} at 97 21.5
 \pinlabel {$z_0$} at 89 43
 
 \pinlabel {$e_2$} at 279 52
 \small
 \pinlabel {$z_0$} at 251 63
 \pinlabel {$z_1$} at 196 87
 \pinlabel {$c_0$} at 228 32
 \pinlabel {$z_0^{-1}$} at 251 37.5

 \pinlabel {$e_1'$} at -6 165
 \small
 \pinlabel {$z_0$} at 20 143
 \pinlabel {$z_1$} at 95 114
 \pinlabel {$c_0$} at 158 148
 \pinlabel {$z_0^{-1}$} at 176 155
 \tiny
 \pinlabel {$z_1^{-1}$} at 203 186
 \small
 \pinlabel {$z_0$} at 173 183
 \pinlabel {$z_1$} at 111 221
 \pinlabel {$c_0$} at 45 173
 \pinlabel {$z_0^{-1}$} at 23 171
 
 \pinlabel {$e_4$} at 114 184
 \small
 \pinlabel {$z_0^{-1}$} at 86 186
 \pinlabel {$z_1^{-1}$} at 112 198
 \pinlabel {$z_0$} at 120 174.5
 \pinlabel {$e_4'$} at 98 152
 \small
 \pinlabel {$z_0^{-1}$} at 121 146
 \pinlabel {$z_1^{-1}$} at 97 138.5
 \pinlabel {$z_0$} at 89 160
 
 \pinlabel {$e_2'$} at 279 169
 \small
 \pinlabel {$z_0$} at 251 180
 \pinlabel {$z_1$} at 196 204
 \pinlabel {$c_0$} at 228 149
 \pinlabel {$z_0^{-1}$} at 251 154.5

\endlabellist
\includegraphics[scale=1.2]{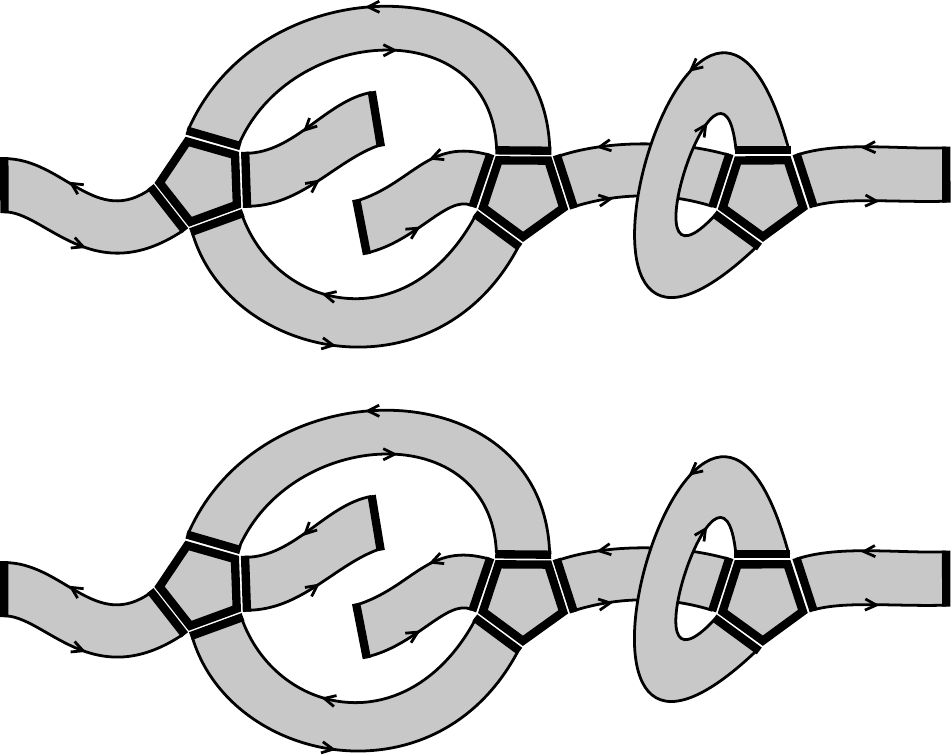}
\caption{Two copies of the partial fatgraph module $B$, 
with rectangle edges labeled as they should be attached 
to $X'$ in the proof of Theorem~\ref{thm:genus_orbifolds}.}
\label{fig:add_one}
\end{center}
\end{figure}
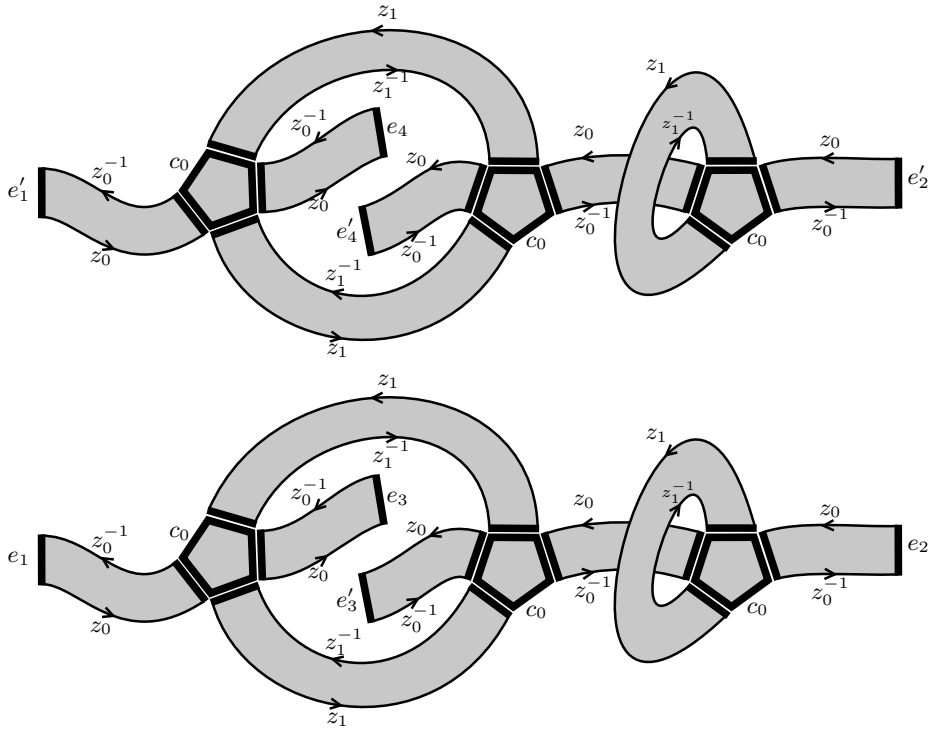

Finally, we show how to build a fatgraph with boundary $wb^{m+k+1}$ 
for any even $k$.  Consider again $Y'$, and recall that 
$Y'$ has at least four unglued polygon edges $pe(z_0)$ and 
$pe(z_0^{-1})$  (it has these unglued edges for every $i$; we choose 
$0$ arbitrarily).  Glue copies of $A_i$ to all unglued polygon edges 
for every index except $0$, and glue a copy of $A_0$ to one of the 
four pairs of unglued edges for index $0$, leaving three pairs.
Next, attach a copy of $A_{0,k}$ to one of the pairs, leaving two pairs.  
Call the resulting fatgraph $X$.
Note $X$ has exactly four unglued 
polygon edges: two $pe(z_0)$, which we denote by $e_1$ and $e_2$ 
and two $pe(z_0^{-1})$, which we denote by $e_3$ and $e_4$.
Let $X'$ be the partial fatgraph which is two copies of $X$, 
and think of $X'$ as a double cover of $X$.  Each $e_s$ 
has two edges covering it, which we denote by $e_s$ and $e_s'$.

We are going to attach two copies of a fatgraph module $B$ to the unglued 
edges in $X'$.  The module $B$ is similar to $A_{0,2}$, and is created from $A_{0,2}$
by removing the $r(z_0)$ rectangle between polygons $0$ and $1$ 
in the construction of $A_{0,2}$ and replacing it with two rectangles, 
one glued to polygon $0$ and one glued to polygon $1$.  This leaves 
four unglued rectangle edges.  It 
is far easier to understand by consulting Figure~\ref{fig:add_one}.  
Though this picture is for a specific example, $B$ 
in any other case is formed by just adding finite-order edges 
and genus loop pairs at locations on the standard polygons; 
it doesn't actually change the form of the module.

Take two copies of $B$, as shown in Figure~\ref{fig:add_one}, 
and attach the edges to $X'$ as labeled to produce a fatgraph $X''$.
Note $X''$ has no unglued infinite-order polygon edges, 
and reading $pe(c_j)$ as $c_j$, we find that the 
boundary of $X''$ is two copies of $wb^{m+k+1}$.  

\begin{figure}[htb]
\begin{center}
\labellist
\small\hair 2pt
 \pinlabel {$e_1$} [ ] at 21 29
 \pinlabel {$e_3$} [ ] at 53 18
 \pinlabel {$e_4$} [ ] at 77 18
 \pinlabel {$e_2$} [ ] at 101 30
 
 \pinlabel {$e_1'$} [ ] at 21 62
 \pinlabel {$e_3'$} [ ] at 53 62
 \pinlabel {$e_4'$} [ ] at 76 62
 \pinlabel {$e_2'$} [ ] at 101 62
\endlabellist
\includegraphics[scale=1.6]{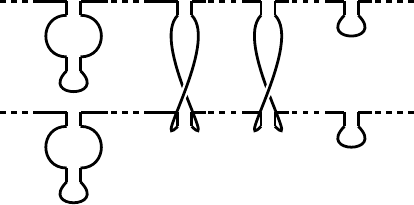}
\caption{A schematic showing the boundary of $X''$ after 
attaching two copies of $B$ to $X'$.  The two levels 
indicate the two sheets of $X'$ as a cover of $X$.
This picture shows why this particular way of attaching the 
edges adds one copy of $b$ in the same place on both boundaries 
of the cover.}
\label{fig:cover_schematic}
\end{center}
\end{figure}
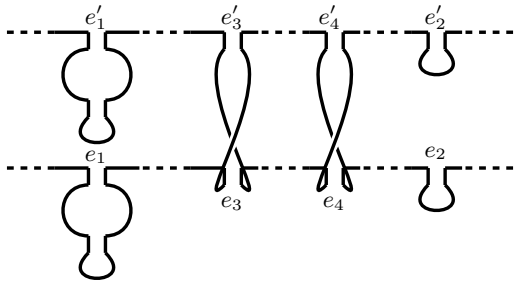

Figure~\ref{fig:cover_schematic} shows a schematic of how the boundary 
behaves after attaching the two copies of $B$.
The exact arrangement of edges used in attaching $B$ to $X'$ is important: 
we need a fatgraph whose boundary is two copies of $wb^{m+k+1}$.
Were we to attach differently, we would have a fatgraph whose 
boundary contained two copies of $w$ and many copies of $b$, but the 
powers of $b$ in between the $w$ might not be the same.
Attaching as instructed places the extra copy of $b$ in the 
same place on both sheets of the cover $X'$.

Now performing the covering trick on $X''$ produces a 
fatgraph $Y$ whose boundary covers $wb^{m+k+1}$, and by construction 
$Y$ satisfies Proposition~\ref{prop:immersion}.  We have now 
shown that $wb^{m+k}$ and $wb^{m+k+1}$ virtually bound 
immersed surfaces for every even $k$, so 
this completes the proof.
\end{proof}

\begin{remark}
Theorem~\ref{thm:genus_orbifolds} applies in the case of a 
hyperbolic surface with a single boundary (and no orbifold points), 
so it resolves \cite{C_faces}, Conjecture~3.16.
\end{remark}

\begin{remark}
When a loop $\gamma$ virtually bounds an immersed surface, it means there 
is an immersed fatgraph with geodesic boundary whose boundary covers $\gamma$ 
with some degree, which we call the \emph{covering degree} of $\gamma$.  
The proofs of Theorems~\ref{thm:disk_orbifolds} 
and~\ref{thm:genus_orbifolds} show that the covering degree of $wb^n$ 
depends on the orders of the finite-order generators and is 
independent of $w$ and $n$.  In particular, if there are no orbifold 
points, then the covering degree is either $1$ or $2$ depending on the 
parity of $n$.
\end{remark}

\end{document}